\documentclass[a4paper,11pt]{article}

\usepackage{amssymb}
\usepackage{amsmath,amsthm}
\usepackage{enumerate,paralist}
\usepackage{amstext}
\usepackage{psfrag}
\usepackage{dsfont}
\usepackage[dvips]{epsfig}
\usepackage[english]{babel}
\usepackage{color}

\usepackage{graphicx}
\usepackage{lipsum}
\usepackage{colortbl}

\theoremstyle{plain}
\newtheorem{theorem}{Theorem}

\newtheorem{lemma}{Lemma}

\theoremstyle{definition}
\newtheorem{definition}{Definition}
\newtheorem{example}{Example}
\newtheorem{remark}{Remark}
\newtheorem{assumption}{Assumption}

\numberwithin{theorem}{section}
\numberwithin{corollary}{section}
\numberwithin{lemma}{section}
\numberwithin{definition}{section}
\numberwithin{example}{section}
\numberwithin{remark}{section}
\numberwithin{proposition}{section}
\numberwithin{assumption}{section}

\newcommand{\Nat}{\mathbb{N}}
\newcommand{\E}{\mathbb{E}}
\newcommand{\ee}{\mathfrak{e}}
\newcommand{\PP}{\mathbb{P}}
\newcommand{\myU}{\widetilde{U}}

\newcommand{\M}{\mathcal{M}}
\newcommand{\Cov}{\mathop{\mathrm{Cov}}\nolimits}
\newcommand{\Var}{\mathop{\mathrm{Var}}\nolimits}
\newcommand{\1}{\mathds 1}
\newcommand{\ffi}{\varphi}
\newcommand{\eps}{\varepsilon}
\newcommand{\myX}{\mathcal{X}}

\newcommand{\MyX}{\widetilde{X}}
\newcommand{\MyV}{\widetilde{V}}
\newcommand{\MyZ}{\widetilde{Z}}
\newcommand{\myA}{\mathcal{A}}

\newcommand{\BB}{\mathcal{B}}
\newcommand{\GG}{\mathcal{G}}
\newcommand{\FF}{\mathcal{F}}
\newcommand{\pp}{\mathcal{P}}

\newcommand{\const}{\text{\rm const}}

\newcommand{\cR}{{\mathbb R}}

\newcommand{\pd}{\partial}

\title{Limit Theorems for the Dynamical Foundation of the Fractional Brownian Motion and Related Models of Anomalous Diffusion
with Random Diffusion Coefficient and 
Time-Dependent Random Hurst parameter}

\author{
Christian Bender\footnote{Saarland University, Department of Mathematics, Campus E2 4, 66123 Saarbr\"ucken, Germany, \texttt{bender@math.uni-saarland.de}}, 
Yana A. Butko\footnote{Kassel University, Institute of Mathematics, Heinrich-Plett-Str. 40, 34132 Kassel, Germany, \texttt{kinderknecht@mathematik.uni-kassel.de}}, 
Mirko D'Ovidio\footnote{
 Sapienza University of Rome,  Department of Basic and Applied Sciences for Engineering,  Antonio Scarpa 16,  00161 Rome,  Italy, \texttt{mirko.dovidio@uniroma1.it}}, 
Gianni Pagnini\footnote{
BCAM--Basque Center for Applied Mathematics, Alameda de Mazarredo 14, 48009
Bilbao, Basque Country, Spain
\& Ikerbasque--Basque Foundation for Science, Plaza Euskadi 5, 48009 Bilbao, Basque Country, Spain, \texttt{gpagnini@bcamath.org}}
}

\begin{document}
\maketitle
\begin{abstract}
Anomalous diffusion is an established phenomenon but still a theoretical challenge in non-equilibrium statistical mechanics.
Physical models are built incrementally,  and
the most recent  and most general family is based on the fractional Brownian motion (fBm) with a random diffusion coefficient (superstatistical fBm) together with a time-dependent random Hurst parameter. We provide here a dynamical foundation for such general family of models.
We consider a dynamical system describing the motion of a test-particle surrounded by $N$ Brownian particles with different masses. This dynamic is governed by underdamped Langevin equations. 
Physical principles of conservation of momentum and energy are met. We prove that, in the limit $N\to\infty$, the test-particle diffuses in time according to a quite general (non-Markovian) Gaussian process whose covariance function is determined by the distribution of the masses of the surround-particles. In particular, with proper choices of the distribution of the masses of the surround-particles, we obtain fBm together with a number of other special cases of interest in modelling anomalous diffusion including time-dependent anomalous exponent. Furthermore, when the ensemble heterogeneity of the surround-particles embodying the environment becomes non-uniform and
joins with the individual inhomogeneity of the test-particles, we show that, in the limit $N\to\infty$, the test-particle diffuses in time according to a quite general conditionally Gaussian process that can be calibrated into a fBm with random diffusion coefficient and random time-dependent Hurst parameter. We conclude our study by reporting the generalised Kolmogorov--Fokker--Planck equations associated to these highly general processes.

    \textbf{Keywords:} 
    fractional Brownian motion, superstatistical fractional Brownian motion, generalized grey Brownian motion, randomly scaled Gaussian processes, random Hurst parameter, time-dependent random Hurst parameter,
        anomalous diffusion,
    ensemble heterogeneity of the surround-particles, individual inhomogeneity of the test-particles, crowded environment,  limit theorems for stochastic processes, generalized evolution equations with
    pseudo-differential operators.
    
\end{abstract}

\section{Introduction}
\textbf{Motivation and goal:}
Experimentally well-established~\cite{klafter_etal-pw-2005,barkai_etal-pt-2012,hofling_etal-rpp-2013,golding_etal-prl-2006}, anomalous diffusion (AD) is a phenomenon observed in many different natural systems belonging to different research fields~\cite{metzler_etal-jpa-2004,kgs,klm}. In particular, AD has become foundational in living systems after a large use of single-particle tracking techniques 
during recent years~\cite{manzo_etal-rpp-2015,shen_etal-cr-2017,simon_etal-nrmp-2024}.
Generally speaking, AD labels all those diffusive processes that are governed by laws that differ from that of classical heat diffusion, 
namely, all those cases when particles' displacements do not accommodate to the Gaussian density function  and/or the variance of such displacements does not grow linearly in time.

In the present paper, 
we establish the physical origin of AD within the picture of 
a test-particle kicked by infinitely many surrounding particles
composing a heterogeneous ensemble.  
In particular, we consider a stochastic dynamical system where 
the microscopic thermal bath is the forcing for the mesoscopic Brownian motion of a bunch of $N$ particles governed by the underdamped Langevin equation; these particles embody the environment surrounding a single test-particle.
Physical conservation principles, 
namely the conservation of momentum and the conservation of energy, 
are met in the considered particle-system in the form of
a coupling between the test-particle and the surround, and 
the fluctuation-dissipation theorem for the motion of the 
surround-particles~\cite{grassia-ajp-2001}, respectively.
The key feature of the present particle-system 
that allows for displaying AD is the extra-randomness that
is introduced both in the form of a distribution of the masses  
of the surround-particles 
- that we call {\it ensemble heterogeneity} -  
and in that of an inner variability among the test-particles 
- that we call {\it individual inhomogeneity}.

\smallskip
\textbf{ Surround-particle mass distribution and the emerging
of the fractional Brownian motion:}
As the first and the main step, 
we introduce the ensemble heterogeneity of the environment
via the distribution of the masses of the surround-particles. 
When the number of mesoscopic Brownian surround-particles $N$ is large enough 
for providing a crowded environment,
then the test-particle displays AD characterised 
by the distribution of the masses of the surround-particles. 
More precisely, we prove that, in the limit $N\to\infty$,
the test-particle diffuses according to a quite general (non-Markovian) Gaussian 
process $(Z_t)_{t \geq 0}$ characterised by a covariance function 
\begin{align}\label{eq:Covariance}
\text{Cov}(Z_t,Z_s) = {\mathcal{D} \,}(v(t)+v(s)-v(|t-s|)) \,,
\end{align}
where $v(\cdot)$ is determined by the distribution of 
the masses of the surround-particles, 
and the constant $\mathcal{D}$ depends on the strength of 
the coupling between the test-particle and the surround-particles. 
Formula~\eqref{eq:Covariance}  
includes the noteworthy case of AD when the variance of the limiting process is (at least asymptotically) a power function whose degree (which is called \emph{anomalous exponent}\footnote{In the case of fractional Brownian motion (fBm) with  \emph{Hurst parameter} $H$, the anomalous exponent is given by $2H$.   The notion \emph{Hurst parameter}  is often used in the literature for description of anomalous exponent also in more general fBm-like models.}) may even change with time. This feature, together with the extra randomnesses due to a random diffusion coefficient and a random Hurst parameter (added in the next steps  below), allows for setting the most general model for AD that in the present formulation is even based on a physically founded stochastic dynamics.

We present  certain choices of the mass distribution of the surround-particles, which lead to several particularly important  AD models.  Specifically, 
we obtain a fractional Brownian motion (fBm) 
with Hurst parameter $H\in(1/2,1)$ as a limiting process $(Z_t)_{t \geq 0}$ using a  suitable power-law distribution of the masses of the surround-particles. 
In this respect, 
we remind that the fBm has experimentally turned out to be 
the underlying stochastic motion in many living systems, 
see, 
inter alia~\cite{magdziarz_etal-prl-2009,szymanski_etal-prl-2009,weiss-pre-2013}.
As another  example, using  a  suitable mixture of  different power-law distributions of the masses of the surround-particles, 
we obtain a mixture of independent fBms with different Hurst parameters as a limiting process $(Z_t)_{t \geq 0}$. This leads to 
a fBm-like 
test-particle's trajectory 
that changes  its 
anomalous exponent with time similarly to recent fBm-based models with time-dependent Hurst exponent 
\cite{slezak_etal-jpa-2023,balcerek_etal-prl-2025}.   
Moreover, we present some distributions of the masses of the surround-particles 
leading to the limiting processes which perform 
a transition from ballistic diffusion to superdiffusion, 
or from ballistic diffusion to classical diffusion. More precisely, the variance functions of these processes can be approximated by power functions whose degrees depend on the time scale and, hence, change with time. Finally, we show how a classical Wiener process can be obtained within our approach.

The Brownian motion of the mesoscopic surround-particles of the system, considered at this step, is   described through 
the underdamped Langevin equation.  
Namely, the velocity is provided
by an Ornstein--Uhlenbeck process and the position is given 
by the integration in time of the velocity according to kinematics. 
Within this setting, the proof of our main limit theorem exploits that, 
conditionally on the masses of surround-particles, 
the dynamics of the test-particle is Gaussian, 
and we can make use of the theory of mixing convergence 
(see, e.g.,~\cite{zbMATH06444973}) to pass to a suitable scaling limit. 
The scaling is, however, worse than in the classical Central Limit Theorem, 
what is compensated by good properties of the Ornstein--Uhlenbeck process.  
The present result pushes forward in a rigorous way 
a preliminary analysis~\cite{dovidio_etal-fcaa-2018} 
aimed to derive  models for AD on the basis of an unspecified (conditionally) 
Gaussian process generated by the superposition of Ornstein--Uhlenbeck processes.

\smallskip
\textbf{ Random diffusion coefficient:}
As a next step, we take into account the individual inhomogeneity of the test-particles. This is reasonable in the case when our test-particle is a complex macromolecule and it may differ from its replicas because of its individual structure features, shape, hydrodynamic radius, etc., even though    
the mass of the test-particle is fixed.  
The individuality of the test-particle is introduced via making 
random the parameter of its coupling with the surrounding environment.  
With this additional assumption we show that, 
in the limit $N\to\infty$, 
the test-particle diffuses according to a process of the form  
$(\sqrt{A} Z_t)_{t\geq0}$, where  $(Z_t)_{t\geq0}$ is a Gaussian process  
with covariance structure as in formula~\eqref{eq:Covariance}. 
The random scaling $\sqrt{A}$ pops up due to the randomness 
of the coupling parameter and plays the role of a 
"random diffusion coefficient".
This result establishes the physical basis 
for the formulation of AD 
within the framework of randomly scaled Gaussian processes (RSGPs). 
In particular, this result justifies the models of
the superstatistical fBm~\cite{
mura_etal-jpa-2008,molina_etal-pre-2016,
mackala_etal-pre-2019,itto_etal-jrsi-2021}, 
where AD is described by a product of a fBm with a proper statistically independent non-negative random variable. The fBm as the Gaussian part provides a non-linear variance in time while the statistically independent non-negative random variable, 
i.e., the random scaling, 
is responsible for a population of diffusion coefficients 
and gives rise to the resulting non-Gaussianity from the realisations 
of the process. 
In this respect, 
we remind that the experimental evidences of a population of 
diffusion coefficients have been reported, for example, in the motion of
mRNA molecules in live E. coli cells \cite{mackala_etal-pre-2019},
$\beta$-adrenergic receptors \cite{grimes_etal-c-2023}
and dendritic cell-specific intercellular adhesion 
molecule 3-grabbing nonintegrin (DC-SIGN) \cite{manzo_etal-prx-2015}.

The superstatistical fBm models are being actively investigated 
in the framework of non-Gaussian analysis such as 
the Grey Noise Analysis by Schneider~\cite{schneider-1990,schneider-1992} 
and its generalisations including 
the Mittag--Leffler Analysis~\cite{grothaus_etal-jfa-2015,grothaus_etal-jfa-2016}
and the Gamma-Grey Analysis~\cite{beghin_etal-pa-2023}.  
The family of superstatistical fBm models includes 
the generalised grey Brownian motion (ggBm) 
\cite{mura-phd-2008,mura_etal-jpa-2008,mura_etal-pa-2008,mura_etal-itsf-2009}, 
gamma-grey Brownian motion
~\cite{beghin_etal-pa-2023}, 
and many other processes related to generalized time fractional 
evolution equations~\cite{bender_etal-fcaa-2022b,bender_etal-fcaa-2022a,
Yana-Merten,pagnini_etal-fcaa-2016,runfola_etal-pd-2024}. 
In particular, models where the random variable $A$
follows the distribution of the square of a 
Weibull distributed variable~\cite{mackala_etal-pre-2019}, 
the generalized Gamma distribution~\cite{sposini_etal-njp-2018} 
or the log-normal distribution~\cite{dossantos_etal-p-2020} 
have been considered in the literature. 
Thus, our results show that 
the superstatistical fBm~\cite{molina_etal-pre-2016,mackala_etal-pre-2019,
itto_etal-jrsi-2021,runfola_etal-rsos-2022,korabel_etal-e-2021,
korabel_etal-sr-2023}
together with its generalization called the diffusing-diffusivity approach~\cite{chubynsky_etal-prl-2014,chechkin_etal-prx-2017,wang_etal-njp-2020,
wang_etal-jpa-2020,sposini_etal-njp-2020,dossantos_etal-csf-2021} 
(here the diffusion coefficient of each test-particle is no longer 
a random variable but a process) stand as promising methods.

\smallskip
\textbf{ Random anomalous exponent:}
Finally, as a third step, 
we allow more randomness in the description of the heterogeneous surround. We introduce an additional random element $H$   and allow the  random masses of the surround-particles to have  a distribution that depends on $H$.
This construction permits to include the case of a random anomalous exponent  into the setting. Namely, the random element $H$ plays the role of the \emph{random Hurst parameter} in the case of fBm-like models.
This framework  is motivated by experimental observations displaying
a distribution of the anomalous exponent 
in the motion of 
histone-like nucleoid-structuring proteins \cite{wang2018}
and quantum dot dynamics in cell cytoplasm
\cite{sabri2020elucidating,janczura2021identifying,cherstvy2019non}.

According to our dynamical system,
the emerging of such behavior can be interpreted as follows. On one hand,
the surround-particles are assumed to be complex molecules 
with their own proper structure features, shape, hydrodynamic radius etc.
Hence, it might be not possible to prepare replicas  
of the environment containing fully identical copies of the surround-particles.
This kind of differences may be expressed through fluctuations of the
Hurst parameter which results to be distributed. 
On the other hand, 
the distribution of the Hurst parameter may express
 a non-uniform ensemble heterogeneity  of the environment such that 
different test-particles can experience  local differences of the surround.
For example, 
if the diffusion of the test-particles is observed for a rather short time, 
then each diffusing test-particle is affected by the environment around its starting position. 
Hence, we have that each test-particle has its own copy of the $N$ "local" 
surround-particles and, each copy of this surround, is characterized
by its own value $h$ of the random element $H$. 
More technically, for a given value $h$ of $H$,  
the masses of $N$ surround-particles are randomly distributed according 
to some distribution $\mu_N(h)$, 
which (in the same style as at the first step) 
leads to the variance $v_h(\cdot)$ of the Gaussian part in the limiting process. 
More precisely, 
having all sources of randomness described above in all three steps, 
we show that, in the limit $N\to\infty$, 
the test-particle diffuses according to a process of the form 
$(\sqrt{A} G^{(H)}_t)_{t\geq0}$. 
Here, for each given value $h$ of $H$, 
the process $(  G^{(h)}_t)_{t\geq0}$ is Gaussian with covariance structure 
as in formula~\eqref{eq:Covariance} 
(with $v_h$ instead of $v$ and with $\mathcal{D}$ depending on $h$).

As a special case, we obtain a randomly scaled (or, superstatistical) 
fBm with random Hurst parameter. Such 
general processes 
with both random diffusion coefficient and random Hurst parameter
have been considered in literature without providing a generation mechanism, 
e.g., \cite{itto_etal-jrsi-2021,lanoiselee_etal-prl-2025}.
Models with random Hurst parameter only have been also considered,
e.g., \cite{randomHurst,
woszczek_etal-c-2025}.
Note that, in the obtained special case,   the random scaling (i.e., the random diffusion coefficient)    turns out to depend both on  random coupling parameter and on random Hurst parameter. The joint distribution of the anomalous exponent and the diffusion coefficient has been experimentally observed, for example, for G proteins \cite{Sungkaworn2017}, intracellular quantum dots \cite{Etoc2018}, $\kappa$-opioid receptor~\cite{Drakopoulos2020},  membrane-less organelles in C. elegans embryos \cite{benelli2021sub}, nanobeads in clawfrog \textit{X. laevis} egg extract~\cite{speckner2021single}, and endosomal dynamics~\cite{korabel_etal-e-2021}. 
More recently, a method for 
deciphering the joint distribution of the anomalous exponent 
and the diffusion coefficient has been proposed together with its
experimental protocol \cite{lanoiselee_etal-prl-2025}.
Further, adapting other examples
of the limiting Gaussian processes of step one to the general setting with random  diffusion coefficient and random Hurst parameter, one may obtain   AD models with time-dependent random anomalous exponent (cf.~\cite{balcerek_etal-prl-2025}). Such models can be interpreted in the following way: Along its trajectory, the test-particle moves from one patch of the domain to another, where each patch is characterized by its own value $h$ of the random Hurst parameter $H$.

To conclude this introductory section,
we report that our study is completed by the derivation of the
Kolmogorov--Fokker--Planck equations associated to our processes,
including the most general case when both the anomalous exponent 
and the diffusion coefficient are random variables.
The derived equation results indeed in a
quite general evolution equation stated 
in terms of pseudo-differential operators. 
As special cases, the derived equation 
contains generalized fractional diffusion equations
that have been studied in literature, see, e.g.,
\cite{bender_etal-fcaa-2022b,
bender_etal-fcaa-2022a,
Yana-Merten,runfola_etal-pd-2024}.

The paper is organised a follows. In Section~\ref{sec:statement}, we state the problem and  present the main limit theorem (together with some special cases) corresponding to the case of Gaussian limiting process. Its detailed proof is reported in  Section~\ref{Proofs}. In Section~\ref{SuperstatFBM}, we generalize our main limit theorem by adding two additional sources of randomness as described in steps two and three above. We discuss  special cases of  a superstatistical fBm with random Hurst parameter and a randomly scaled fBm-like model with time-dependent random Hurst parameter. Moreover, we outline what kind of evolution equations serve as Kolmogorov--Fokker--Planck equations for such limiting processes. 

\section{Fractional Brownian Motion and Other Gaussian Processes as a Result of 
Ensemble Heterogeneity of the Surround-Particles}
\label{sec:statement}
In this section, we consider AD 
caused by the heterogeneity of the surround-particles that compose the environment where 
the diffusion of the test-particle 
takes place. Hence 
we consider the motion of a single test-particle with mass $M$  
that is immersed into a surround composed by a heterogeneous ensemble of 
{$N$} 
Brownian particles with positive masses $m_{k,N}$, $k=1,\ldots,N$, 
positions $Y^{k,N}_t$ and velocities $U^{k,N}_t$, 
{$N\in\Nat$}. 
Let $X^N_t$ be the position and $V^N_t$ be the velocity of the test-particle. 
This system of particles is described 
by the following system of Langevin equations:
\begin{align}\label{original system}
\left\{ \begin{array}{ll}
dX^N_t=V^N_tdt \,, & X^N_0={ 0} \,,\\
dV^N_t=\frac1M
\sum_{k=1}^N\left(\beta_{k,N}U^{k,N}_tdt-\alpha_{k,N}V^N_tdt \right) \,, 
& V^N_0={ 0} \,,\\
dY^{k,N}_t=U^{k,N}_tdt \,, & Y^{k,N}_0=y^{k,N}_0 \,,\\
{dU^{k,N}_t=
F(U^{k,N}_t,V^N_t,
m_{k,N},\alpha_{k,N},\beta_0^{k,N},\beta_{k,N}) dt
+\frac{\sqrt{2\sigma_0^{k,N}}}{m_{k,N}}dW^k_t} \,, & U^{k,N}_0=u^{k,N}_0 \,,\\
& k=1,\ldots,N \,,
\end{array} \right.
\end{align}
where $(W^1_t)_{t\geq0},\ldots, (W^n_t)_{t\geq0},\ldots$ 
is a sequence of independent Wiener processes 
on some probability space $(\Omega,\mathcal{F},\mathbb{P})$, 
$\alpha_{k,N}$, 
{$\beta_0^{k,N}$}, $\beta_{k,N}$, 
$k=1,\ldots,N$, 
{$N\in\Nat$}, 
are positive coupling constants 
{and $\sigma_0^{k,N} > 0$ is the noise amplitude.}

{In view of establishing the physical origin of AD, 
we set system \eqref{original system} such that conservation principles
are met, that is $i)$ the conservation of momentum 
and $ii)$ the conservation of energy,
this last in the form
of the fluctuation-dissipation theorem for the diffusion of 
the surround-particles. 
Therefore, we have  for all $k=1,\ldots,N $
\begin{equation}\label{eq:F}
F(U^{k,N}_t,V^N_t,
m_{k,N},\alpha_{k,N},\beta_0^{k,N},\beta_{k,N})
=-\frac{\beta_0^{k,N}+\beta_{k,N}}{m_{k,N}}U^{k,N}_t
+\frac{\alpha_{k,N}}{m_{k,N}}V^N_t \,,
\end{equation}
and the fluctuation-dissipation theorem implies 
\begin{align}\label{FDT}
\left(\beta_0^{k,N}+\beta_{k,N}\right)\kappa_BT
=\sigma_0^{k,N},\qquad k=1,\ldots,N \,, \quad N\in\Nat \,,
\end{align}
where $\kappa_B$ and $T$ are the Boltzman constant and 
the temperature, respectively.} 	
Furthermore, we impose the following assumption:
	
\begin{assumption}\label{asmp:hyp1}
We assume that there exists a constant $\sigma>0$ such that for any $N\in\Nat$ and any $k=1,\ldots,N$,
\begin{align}\label{Hypothesis 1} 
\frac{\sqrt{\sigma_0^{k,N}}}{m_{k,N}}=\sqrt{\sigma} \,.
\end{align}
\end{assumption}
Therefore,  
combining~{  \eqref{eq:F}}, ~\eqref{FDT} and~\eqref{Hypothesis 1} 
and using $\gamma:=\frac{\sigma}{\kappa_BT}$,  
we obtain the followng system of equations:
	
	\begin{align}\label{final system}
		\left\{ \begin{array}{ll}
			dX^N_t=V^N_tdt \,, & X^N_0={ 0} \,,\\
			dV^N_t=\frac1M\sum_{k=1}^N\left(\beta_{k,N}U^{k,N}_tdt-\alpha_{k,N}V^N_tdt  \right) \,, & V^N_0={ 0} \,,\\
			dY^{k,N}_t=U^{k,N}_tdt \,, & Y^{k,N}_0=y^{k,N}_0 \,,\\
			dU^{k,N}_t=-\gamma m_{k,N}U^{k,N}_tdt+\frac{\alpha_{k,N}}{m_{k,N}}V^N_tdt+\sqrt{2\sigma}dW^k_t \,, & U^{k,N}_0=u^{k,N}_0 \,,\\
			& k=1,\ldots,N \,.
		\end{array}  \right.
	\end{align}
	
	For each fixed $t_0>0$, we are interested in the behaviour of the system~\eqref{final system} during the time interval $[0,t_0]$ in the limit $N\to\infty$.
	In  order to streamline the notation,
	we write  $g_1 \simeq g_2$ if and only if $g_1(N)/g_2(N) \to 1$ as $N\to \infty$,
	$g_1 \propto g_2$ if and only if $\exists \, C>0$ such that $g_1 \simeq C \, g_2$, and $g_1\lesssim g_2$ if and only if $\exists \, C>0$ such that $g_1 \leq C \, g_2$ for all $N\in\Nat$.

Below we present a list of assumptions on parameters of system~\eqref{final system} which control the behaviour of all parameters in the limit $N\to\infty$ such that AD is generated.
These assumptions are the key-feature for 
driving the emerging AD to belong to a desired
special family of anomalous processes.
Actually, to this purpose, 
the main element is the function $v$ in Assumption~\ref{asmp:m} which prescribes the variance of the limiting AD process.

\begin{assumption}\label{asmp:coef}
		For  given parameters  $a>0$, $b>0$, $C_\alpha>0$, $C_\beta>0$,
		we assume that 
		{ \begin{align*}
				& \alpha_{k,N}\simeq C_\alpha N^{-a},\qquad k=1,\ldots,N,\\
				& \beta_{k,N}\simeq C_\beta N^{-b},\qquad k=1,\ldots,N.
			\end{align*} 
		}
\end{assumption}
\begin{assumption}\label{asmp:m}
(i)	For each fixed $N\in\Nat$, we assume that $m_{1,N},\ldots,m_{N,N}$ are 
		i.i.d. random variables  on the same probability space $(\Omega,\mathcal{F},\mathbb{P})$ with distribution $\mu_N$  supported in $[m^N_{min},\infty)\subset (0,\infty)$. We suppose that
		$$
		m^N_{min}\simeq  N^{-d}\quad \text{for some }\, d\in \mathbb{R}
		$$
and there exist constants  $\delta\geq 0$ and $C_\delta>0$, a sequence $(m^*_N)_{N\in\Nat}$ and a nondecreasing function $\dot v:[0,t_0]\rightarrow [0,\infty)$ such that  
$$
m^*_N\simeq C_{\delta}N^{-\delta} \qquad \text{and}\qquad 
\frac{\ee_N(t)}{m^*_N}\uparrow \dot v(t),\quad t\in [0,t_0].
$$
Here,
$$
\ee_N(t):= \int_{(0,\infty)} \frac{1-e^{-\gamma yt}}{y^2} \mu_N(dy),
$$
and we will use the notation 
$$
v(t):=\int_0^t \dot v(\tau)d\tau.
$$
 Moreover, we assume that there exists $d'\in \mathbb{R}$ such that
$$
\int_{(0,\infty)} y^{-4} \mu_N(dy) \lesssim N^{d'}.
$$
\noindent (ii) Furthermore, let $\mathcal{M}:=\left(m_{1,N},\ldots,m_{N,N}\right)_{N\in\Nat}$ denote the collection of all random masses. 
We assume that  $\mathcal{M}$ is a family of independent random variables and  that all random variables $m_{k,N}$, $N\in\Nat$, $k=1,\ldots,N$, are independent also of the Wiener processes $(W^1_t)_{t\geq0},\ldots, (W^{n}_t)_{t\geq0},\ldots$.  
\end{assumption}
 Note that, since $m_{k,N}$ are random, so are $\sigma_0^{k,N}$ and $\beta_0^{k,N}$ due to Assumption~\ref{asmp:hyp1} and equality~\eqref{FDT}.
\begin{assumption}\label{assmp:m extra}
For a part of our results, we will use also the following additional assumptions on the distribution $\mu_N$. They are either 
\begin{enumerate}
\item[(i)] There exist $\eps> 0$ and $C=C(\eps,t_0)>0$ such that the variance function $v(t)$ satisfies
\begin{align}\label{eq:assmp:v}
v(t)\leq C t^{1+\eps},\qquad t\in[0,t_0].
\end{align}
Moreover,  there exist $\epsilon\in(0,1)$, $C=C(\epsilon)>0$  and $N_0\in\Nat$ such that for every $N>N_0$
\begin{align}\label{eq:assmp:muN}
\frac{1}{m^*_N}\left(\int_{(0,\infty)}\left[ \frac{\gamma}{y} {\1}_{\{0<y\leq 1\}}+\frac{\gamma^\epsilon}{y^{2-\epsilon}}  {\1}_{\{y> 1\}} \right]\mu_N(dy) \right)< C.
\end{align}
\end{enumerate}
or
\begin{enumerate}
\item[(ii)] The masses $m_{k,N}$, $k=1,\ldots,N$, $N\in\Nat$, are deterministic, i.e., the distribution $\mu_N$ is the Dirac measure concentrated at some point on $(0,\infty)$; this point may depend on $N$. Moreover, there exists $C=C(t_0)>0$ such that
\begin{align}
v(t)\leq Ct,\qquad t\in[0,t_0].
\end{align}
\end{enumerate}
\end{assumption}

\begin{assumption}\label{assmp:u0}
We assume that $\left(u_0^{k,N},\, N\in\Nat,\, k=1,\ldots,N\right)$ is a family of  independent  random variables which are independent also of  all Wiener processes   $(W^1_t)_{t\geq0}$,..., $(W^{n}_t)_{t\geq0},\ldots$. And, given $\M$,  $u_0^{k,N}$ has Gaussian distribution $\mathcal{N}(0,\sigma/(\gamma m_{k,N}))$, $k=1,\ldots,N$.
\end{assumption}
	
\begin{assumption}\label{assmp:coef2}
We pose the following conditions on the parameters:
\begin{align*}
			\left\{\begin{array}{l}
				{ 0<a<1,\, b>0},\,{ d, d' \in \mathbb{R},\, \delta\geq 0},\\
				{ 2(a-b)-\delta=1},\\
				d'<5+8(b-a)\\
				{b>d}.
			\end{array}
			\right.
\end{align*}
\end{assumption}

The main result of this paper is the following theorem:	
\begin{theorem}\label{Thm}
(i) Fix any $t_0>0$. Under Assumptions~\ref{asmp:hyp1},~\ref{asmp:coef},~\ref{asmp:m},~\ref{assmp:u0},~\ref{assmp:coef2},
consider  a centered Gaussian process $(Z_t)_{t\in[0,t_0]}$ 
 with 
 covariance function
\begin{align*}
\Cov(Z_t,Z_s)=\left(\frac{2\sigma C^2_\beta C_\delta}{\gamma^2 C^2_\alpha}\right)\frac12\left(v(t)+v(s)-v(|t-s|)  \right).
\end{align*}
Then the processes $\left(X^N_t\right)_{t\in[0,t_0]}$ in~\eqref{final system}, $N\in\Nat$, converge as $N\to\infty$ to $(Z_t)_{t\in[0,t_0]}$  in finite dimensional distributions.

\smallskip

\noindent (ii) If, additionally, Assumption~\ref{assmp:m extra}~(i) or  Assumption~\ref{assmp:m extra}~(ii) is true,  the  processes $\left(X^N_t\right)_{t\in[0,t_0]}$ in~\eqref{final system}, $N\in\Nat$, converge as $N\to\infty$ to $(Z_t)_{t\in[0,t_0]}$  in distribution   on the space ${C}[0,t_0]$ of continuous functions from $[0,t_0]$ to $\mathbb{R}$.
\end{theorem}

{
\begin{remark}
Note that the process $(Z_t)_{t\in[0,t_0]}$  has stationary increments. Indeed, $\E[Z_{t+h}-Z_t]=0$, $\Var(Z_{t+h}-Z_t)=2\mathcal{D}v(|h|)$ with $\mathcal{D}:=\frac{\sigma C^2_\beta C_\delta}{\gamma^2 C^2_\alpha} $ and, since the process  $(Z_t)_{t\in[0,t_0]}$ is Gaussian, the distribution of $Z_{t+h}-Z_t$  does  not depend on $t$.  
\end{remark}
}

We  present now some important special cases when the Assumptions~\ref{asmp:coef},~\ref{asmp:m}~(i),~\ref{assmp:m extra},~\ref{assmp:coef2} are satisfied and then outline the ideas of the proof.  The rigorous proof of Theorem~\ref{Thm} with all technical details is given in  Section~\ref{Proofs}.

\begin{example}\label{ex:case1}
Suppose $\nu$ is the L\'evy measure of a pure jump subordinator with Laplace exponent 
$\Phi(\lambda)=\int_{(0,\infty)} (1-e^{-\lambda y})\nu(dy)$.  We assume that $\nu$ is continuous and $\int_{(1,\infty)} y^2 \nu(dy)=\infty$. Now, fix $d,\delta>0$ and define  sequences $(m^N_{max})_{N\in\Nat}$ and $(m^N_{min})_{N\in\Nat}$ via $m^N_{min}:=N^{-d}$ and $\int_{(0,m^N_{max}]} y^2\nu(dy)  = N^{\delta}$. We consider the distribution $\mu_N$ given by
\begin{align}\label{eq:muN}
\frac{d \mu_N}{d\nu}(y)=m^*_N y^2{\1}_{(m^N_{min},m^N_{max}]}(y),
\end{align}
where the normalizing constant $m^*_N$ turns $\mu_N$ into a probability measure. Then, $m^*_N\simeq N^{-\delta}$ and for any $t_0>0$ and any $t\in[0,t_0]$
\begin{equation*}
\frac{\ee_N(t)}{m^*_N}=  \int_{(m^N_{min},m^N_{max}]} (1-e^{-\gamma yt})\; \nu(dy) \uparrow  \Phi(\gamma t)=:\dot v(t).
\end{equation*}
Note that $\Phi(\cdot)$ is  nonnegative and nondecreasing as a  Bernstein function. The limiting variance function is then
$$
v(t)=\int_0^t \Phi(\gamma \tau) d\tau.
$$ 
Moreover, since the L\'evy measure of any subordinator satisfies the condition $\int_{(0,\infty)} 1\wedge y \,\nu(dy)<\infty$ (here $\wedge$ stands for minimum), we have
\begin{align*}
\int_{(0,\infty)} y^{-4} \mu_N(dy)
&= m^*_N \int_{(m^N_{min},m^N_{max}]} y^{-2} \nu(dy) \\
&
\leq  m^*_N\left((m^N_{min})^{-3} \int_{(m^N_{min},1]} y \nu(dy) + \int_{(1,\infty)} \nu(dy)  \right) \lesssim N^{3d-\delta}.
\end{align*}
Hence, we may take $d'=3d-\delta$. And Assumption~\ref{asmp:m}~(i) is satisfied.  Then, Assumption~\ref{assmp:coef2} is satisfied in the following situation:
$$
b\in (0,1/2), \quad  a\in (b+1/2,(b+2/3)\wedge 1), \quad \delta=2(a-b)-1,\quad d<(4/3-2(a-b))\wedge b.
$$
Furthermore, condition~\eqref{eq:assmp:muN} of Assumption~\ref{assmp:m extra} has then the following view:
\begin{align}\label{eq:cond nu case 1}
\exists\,\epsilon\in(0,1)\,\,\text{and}\,\, C=C(\epsilon)>0\,:\qquad\int_{(1,\infty)} y^\epsilon   \nu(dy)\leq C,
\end{align}
since the condition $\int_{(0,1]}y\nu(dy)<\infty$ is satisfied by the L\'evy measure of any subordinator.

\bigskip

Let us consider now some special cases of the above setting:

\bigskip

\noindent (i) \textbf{Fractional Brownian motion:}
Let $H\in (1/2,1)$ and let $\nu$ be the L\'evy measure of the $(2H-1)$-stable subordinator, i.e., 
$$
\nu(dy):=\1_{(0,\infty)}(y)y^{-2H}dy.
$$
We define $\mu_N$ in accordance with~\eqref{eq:muN}. Note that $m^N_{max}=(3-2H)^{\frac{1}{3-2H}} N^{\frac{\delta}{3-2H}}$.
Since
$$
\Phi(\lambda)=\frac{\Gamma(2-2H)}{2H-1} \lambda^{2H-1},
$$
 we find the variance function
 $$
 v(t)=\frac{\Gamma(2-2H) \gamma^{2H-1}}{2H(2H-1)} t^{2H}.
 $$
Hence, up to a multiplicative factor, the limiting process is a fractional Brownian motion with Hurst parameter $H$. Condition~\eqref{eq:assmp:v}  is satisfied, e.g., with $\eps:=2H-1$ and $C:=\frac{\Gamma(2-2H) \gamma^{2H-1}}{2H(2H-1)}$; condition~\eqref{eq:cond nu case 1} is satisfied with any $\epsilon\in(0, 2H-1)$. Hence, Assumption~\ref{assmp:m extra}~(i) is  satisfied.

\bigskip

\noindent (ii) 
\textbf{Time-dependent anomalous exponent
from a mixture of power-law distributions of the surround-masses:
A mixture of fBm's  with different Hurst parameters:}
Let $K\in\Nat$ and $H_1<\ldots<H_K$ with $H_1, H_K\in(1/2,1)$. Let 
$$
\nu(dy):=\1_{(0,\infty)}(y)\sum_{k=1}^K y^{-2H_k}\,dy.
$$
We define $\mu_N$ in accordance with~\eqref{eq:muN}. Then 
$$
\Phi(\lambda)= \sum_{k=1}^K \frac{\Gamma(2-2H_k)}{2H_k-1} \lambda^{2H_k-1},
$$
and we find the variance function
 $$
 v(t)=\sum_{k=1}^K \frac{\Gamma(2-2H_k) \gamma^{2H_k-1}}{2H_k(2H_k-1)} t^{2H_k}.
 $$
Hence,  the limiting process is a sum of  independent fBm's  with Hurst parameters $H_1,\ldots,H_K$ (up to  multiplicative factors). In this case, the character of anomalous diffusion changes with time: $v(t)\approx  \frac{\Gamma(2-2H_1) \gamma^{2H_1-1}}{2H_1(2H_1-1)} t^{2H_1}$ for small $t$ and $v(t)\approx  \frac{\Gamma(2-2H_K) \gamma^{2H_K-1}}{2H_K(2H_K-1)} t^{2H_K}$ for large $t$. Assumption~\ref{assmp:m extra}~(i) is satisfied, e.g., with $\eps:=2H_1-1$ and $\epsilon\in(0,2H_1-1)$.

\bigskip

\noindent (iii) Further couples $(\nu,\Phi)$ satisfying the setting (cp.~\cite{BF}):
\begin{align*}
& \nu(dy):=\frac{(2-y)e^{-1/y}+y}{2\sqrt{\pi}y^{5/2}}dy\qquad\text{and}\qquad \Phi(\lambda):=\sqrt{\lambda}\left(1-e^{-2\sqrt{\lambda}}  \right);\\
&
\nu(dy):=\frac{1-e^{-y}(1+y)}{y^2}dy\qquad\text{and}\qquad\Phi(\lambda):=\lambda\log(1+1/\lambda).
\end{align*}
\end{example}

\begin{example}\label{ex:case1b}
Suppose $\nu$ is the L\'evy measure of a pure jump subordinator with Laplace exponent $\Phi(\lambda)=\int_{(0,\infty)} (1-e^{-\lambda y})\nu(dy)$.  We now assume that  $\int_{(1,\infty)} y^2 \nu(dy)<\infty$. Let $m^N_{min}=N^{-d}$ for some $d>0$ and consider the distribution $\mu_N$ given by
\begin{align}\label{eq:muN case b}
\frac{d \mu_N}{d\nu}(y)=m^*_N y^2{\1}_{(m^N_{min},\infty)}(y),
\end{align} 
where the normalizing constant $m^*_N$ turns $\mu_N$ into a probability measure. Therefore, $m^*_N\simeq (\int_{(0,\infty)} y^2 \nu(dy))^{-1}$, i.e., $\delta=0$. Moreover
\begin{equation*}
\frac{\ee_N(t)}{m^*_N}:=  \int_{(m^N_{min},\infty)} (1-e^{-\gamma yt})\; \nu(dy) \uparrow  \Phi(\gamma t)=:\dot v(t),
\end{equation*}
leading, as before, to the limiting variance function
$$
v(t)=\int_0^t \Phi(\gamma \tau) d\tau.
$$ 
A similar reasoning as before shows
\begin{align*}
\int_{(0,\infty)} y^{-4} \mu_N(dy)\leq m^*_N (m^N_{min})^{-3} \int_{(m^N_{min},\infty)} y \nu(dy) \lesssim N^{3d}.
\end{align*}
Hence, we may take $d'=3d$. And Assumption~\ref{asmp:m}~(i) is satisfied. Assumption  \ref{assmp:coef2} is, then, satisfied in the following situation:
$$
b\in (0,1/2), \quad  a=b+1/2, \quad  d<(1/3)\wedge b.
$$
Condition~\eqref{eq:assmp:muN} of Assumption~\ref{assmp:m extra} has again the form~\eqref{eq:cond nu case 1} and is satisfied for any $N\in\Nat$ and any $\epsilon\in(0,1]$.

\bigskip

Let us consider now some special cases of the above setting:

\bigskip

\noindent (i) \textbf{Transition from ballistic diffusion to superdiffusion:} Let $H\in (1/2,1)$ and let $\nu$ be the L\'evy measure of a tempered $(2H-1)$-stable subordinator, i.e., 
$$
\nu(dy):=\1_{(0,\infty)}(y)e^{-y} y^{-2H}dy.
$$
Hence, $\int_{(0,\infty)}y^2\,\nu(dy)<\infty$. We define $\mu_N$ in accordance with~\eqref{eq:muN case b}. Then (cp.~\cite{BF}),
$$
\Phi(\lambda)=(\lambda+1)^{2H-1}-1.
$$
Hence, we obtain the variance function $v(t)=\frac{(\gamma t+1)^{2H}-1}{2H\gamma }-t$. In this case, the character of anomalous diffusion changes with time:  $v(t)\approx C t^2$ for small $t$ and $v(t)\approx C t^{2H}$ for large $t$. Condition~\eqref{eq:assmp:v} is then satisfied, e.g., with $\eps:=2H-1$. Hence, Assumption~\ref{assmp:m extra}~(i) is true.

\bigskip

\noindent (ii) \textbf{Transition from ballistic diffusion to classical diffusion:} Let   $\nu$ be the following L\'evy measure: 
$$
\nu(dy):=\1_{(0,\infty)}(y)\gamma e^{-\gamma y}dy.
$$
Hence, $\int_{(0,\infty)}y^2\,\nu(dy)<\infty$. We define $\mu_N$ in accordance with~\eqref{eq:muN case b}. Then (cp.~\cite{BF}),
$$
\Phi(\lambda)=\frac{\lambda}{\lambda+\gamma}.
$$
Hence, we obtain the variance function $v(t)=t-\log(t+1)$. In this case, the character of anomalous diffusion changes with time:  $v(t)\approx C t^2$ for small $t$ and $v(t)\approx  t$ for large $t$. Therefore, condition~\eqref{eq:assmp:v} is satisfied for any $\eps\in(0,1]$ with some suitable constant $C=C(\eps,t_0)>0$.

\bigskip

\noindent (iii) Further couples $(\nu,\Phi)$ satisfying the setting (cp.~\cite{BF}):
\begin{align*}
& \nu(dy)=\1_{(0,\infty)}(y)e^{-y}y^{-1}dy,\qquad \Phi(\lambda)=\log(1-\lambda);\\
&
\nu(dy)=\1_{(0,{\infty})}(y)\frac{1}{\Gamma(1-\alpha)}e^{-y}y^{-\alpha}dy,\qquad\Phi(\lambda)=1-(1+\lambda)^{\alpha-1},\quad \alpha\in(0,1).
\end{align*}
\end{example}

\begin{example}\label{Example:BM}
In this example, the masses $m_{k,N}$, $k=1,\ldots,N$, $N\in\Nat$, are deterministic.

\bigskip

\noindent (i) \textbf{Classical diffusion:}
Fix some $\delta>0$ and define $\mu_N$ to be the Dirac measure concentrated at $N^{\delta/2}$. Then,
$$
N^{\delta} \ee_N(t)= \int_{(0,\infty)} \frac{1-e^{-\gamma yt}}{y^2} \mu_N(dy)=  1-e^{-\gamma N^{\delta/2}t}\uparrow 1_{(0,t_0]}(t)
$$
leading to $v(t)=t$, and, hence, to a   Wiener process as limiting process. Hence, Assumption~\ref{assmp:m extra}~(ii) is satisfied. Further, we may take $d=-\delta/2<0$ and $m^*_N=N^{-\delta}$. Moreover,
\begin{align*}
\int_{(0,\infty)} y^{-4} \mu_N(dy)= N^{-2\delta},
\end{align*}
i.e., $d'=-2\delta<0$. With such choice of parameters, Assumption~\ref{asmp:m}~(i) is satisfied.  Assumption~\ref{assmp:coef2} is, then, satisfied in the following situation:
$$
b\in (0,1/2), \quad  a\in (b+1/2,(b+3/4)\wedge 1), \quad \delta=2(a-b)-1.
$$

\bigskip

\noindent (ii)  \textbf{Transition from ballistic diffusion to classical diffusion:} Let now the masses $m_{k,N}$, $k=1,\ldots,N$, $N\in\Nat$ be deterministic and do not change with $N$, say, $\mu_N$ is the Dirac measure concentrated at $1$.  Then $m^N_{min}=1$, $d=0$,  
$$
 \ee_N(t)= \int_{(0,\infty)} \frac{1-e^{-\gamma yt}}{y^2} \mu_N(dy)=  1-e^{-\gamma  t}=\dot{v}(t),
$$
and hence $\delta=0$, $m^*_N=1$, $d'=0$. With such choice of parameters, Assumption~\ref{asmp:m}~(i) is satisfied. Assumption~\ref{assmp:coef2} is, then, satisfied in the following situation:
$$
b\in(0,1/2),\quad a = b+1/2.
$$
The variance function is  $v(t)=t+\frac1\gamma\left( e^{-\gamma t}-1\right)$. In this case, the character of anomalous diffusion changes with time:  $v(t)\approx C t^2$ for small $t$ and $v(t)\approx  t$ for large $t$. Hence, Assumption~\ref{assmp:m extra}~(ii) is satisfied.
\end{example}	

\medskip

Below we outline the ideas and the structure of the proof of Theorem~\ref{Thm}. A complete rigorous proof can be found in  Section~\ref{Proofs}.

\medskip	
	
\noindent \textbf{Step 1}: We show that we may neglect the drift terms $\frac{\alpha_{k,N}}{m_{k,N}}V^N_tdt$ in the system~\eqref{final system} when $N\to\infty$. Hence, we may replace the processes $(U^{k,N}_t)_{t\in[0,t_0]}$ in the system~\eqref{final system} by the corresponding Ornstein-Uhlenbeck processes $(\widetilde{U}^{k,N}_t)_{t\in[0,t_0]}$,
	\begin{align}\label{eq:OUprocess}
		\myU^{k,N}_t:=u^{k,N}_0 e^{-\gamma m_{k,N}t}+\sqrt{2\sigma}\int_0^t e^{-\gamma m_{k,N}(t-s)}dW^k_s,\qquad k=1,\ldots,N.
	\end{align}
Therefore, solving the ODEs in the first two lines of  the system~\eqref{final system}, we may approximate the position of the test-particle  as $N\to\infty$  by the process  $\left(\MyX^N_t\right)_{t\in[0,t_0]}$, given by
	\begin{align}\label{eq:MyX}
		\MyX^N_t:=\int_0^t\MyV^N_\tau d\tau:=\frac{1}{M}\int_0^t e^{-\frac{\myA_N}{M}s}\sum_{k=1}^N\beta_{k,N}\int_s^t\myU^{k,N}_{\tau-s}d\tau ds,\qquad t\in[0,t_0],
	\end{align}
	where $\myA_N:=\sum_{k=1}^N\alpha_{k,N}$.
	\medskip
	
\noindent \textbf{Step 2}: We show that, as $N\to\infty$,  the process $\left(\MyX^N_t\right)_{t\in[0,t_0]}$  can be approximated by  the process $\left(\MyZ^N_t\right)_{t\in[0,t_0]}$, where
	\begin{align}\label{eq:MyZ}
		\MyZ^N_t:= \frac{1}{\myA_N}\sum_{k=1}^N\beta_{k,N}\int_0^t \myU^{k,N}_\tau d\tau. 
	\end{align}

	\medskip
	
	\noindent \textbf{Step 3}: We show that 
	the processes $\left( \MyZ^N_t \right)_{t\in[0,t_0]}$, $N\in\Nat$, converge in a suitable sense  to the  process $\left(Z_t   \right)_{t\in[0,t_0]}$ as in the statement of Theorem~\ref{Thm}. Finally, combining the results of all three steps, we obtain the statement of the Theorem.
	
\begin{remark}	Pay attention, that the scaling coefficients $\frac{\beta_{k,N}}{\myA_N}\simeq C N^{a-b-1}$.  In the case $\delta=0$, this scaling correspond to the scaling in the classical Central Limit Theorem. In the case $\delta>0$, this scaling is  worse than that one  in the Central Limit Theorem since $a-b-1>-1/2$ due to Assumption~\ref{assmp:coef2}. This is however compensated by the good properties of Ornstein-Uhlenbeck processes $(\widetilde{U}^N_t)_{t\in[0,t_0]}$. 
\end{remark}

\section{Anomalous Diffusion as a Result of 
both Ensemble Heterogeneity of the Surround-Particles and Individual Inhomogeneity of
the Test-Particles.
}\label{SuperstatFBM}

In this Section, 
we investigate which family of AD 
can be obtained by the same approach as in Section~\ref{sec:statement} combining now 
ensemble heterogeneity of the surround-particles
with the individual inhomogeneity of
the test-particles.
We consider again the same system of particles as before, namely, system~\eqref{final system} and investigate its behaviour on a time interval $[0,t_0]$ when $N\to\infty$. However, we modify our assumptions in the following way:

\begin{assumption}\label{asmp:coef new}
We assume that the diffusing test-particle is a complex molecule with its own individual  structure features, shape, hydrodynamic radius etc. This can be considered as an intrinsic individual inhomogeneity of the diffusing particles, i.e., even if we would consider several test-particles of the same kind and with the same mass $M$, the coupling of each one to the surround is individual. We describe this intrinsic 
individual inhomogeneity by a random variable  $C_\alpha$ on the same probability space $(\Omega,\FF,\PP)$, which is strictly positive and independent of all the Wiener processes $(W^1_t)_{t\geq0},\ldots, (W^{n}_t)_{t\geq0},\ldots$. Let $A:=(C_\alpha)^{-2}$. 
We assume that  
$$
\E[A]<\infty.
$$
  Furthermore, for  given parameters  $a>0$, $b>0$, $C_\beta>0$, we assume that 
\begin{align*}
& \alpha_{k,N}   = C_\alpha N^{-a},\qquad k=1,\ldots,N,\\
&
 \beta_{k,N}\simeq C_\beta N^{-b},\qquad k=1,\ldots,N.
\end{align*} 
		
\end{assumption}
\begin{assumption}\label{asmp:m new}
\noindent  (i) We consider a random vector $H$ on the probability space $(\Omega,\FF,\PP)$ which is independent of $A$ and of all Wiener processes $(W^1_t)_{t\geq0},\ldots, (W^{n}_t)_{t\geq0},\ldots$, and takes values in a Borel-measurable subset $\mathcal{H}$ of $\mathbb{R}^K$ (for some $K\in\Nat$), equipped with the Borel-$\sigma$-field  $\mathcal{B}(\mathcal{H})$.

\medskip

\noindent(ii) 
 We assume that, conditionally on $H$, the random masses $m_{k,N}$, $k=1,\ldots,N$, for each fixed $N\in \Nat$ are i.i.d. random variables  on the same probability space $(\Omega,\mathcal{F},\mathbb{P})$ and that, for a fixed regular version $\mu_N:\mathcal{H}\times \mathcal{B}(\mathbb{R})\rightarrow [0,1]$ of the  conditional distribution of $m_{k,N}$ on $H$, $\mu_N(h,\cdot)$ is 
 supported in $[m^N_{min},\infty)\subset (0,\infty)$, where the constant $m^N_{min}$  does not depend on $h\in \mathcal{H}$.

\medskip

\noindent(iii)   We assume that
		$$
		m^N_{min}\simeq  N^{-d}\quad \text{for some }\, d\in \mathbb{R}
		$$
and there exist constants  $\delta\geq 0$ and $C_\delta>0$ and a sequence $(m^*_N)_{N\in\Nat}$, satisfying $m^*_N\simeq C_{\delta}N^{-\delta}$, such that  for each fixed  $h\in\mathcal{H}$ there exists   a nondecreasing function $\dot v_h:[0,t_0]\rightarrow [0,\infty)$ satisfying 
$$
\frac{\ee^{(h)}_N(t)}{m^*_N}\uparrow \dot v_h(t),\quad t\in [0,t_0].
$$
Here,
$$
\ee^{(h)}_N(t):= \int_{(0,\infty)} \frac{1-e^{-\gamma yt}}{y^2} \mu_N(h,dy),
$$
and we will use the notation 
$
v_h(t):=\int_0^t \dot v_h(\tau)d\tau.
$
 Moreover, we assume that there exists $d'\in \mathbb{R}$ and $C>0$, such that for all  $h\in\mathcal{H}$
$$
\int_{(0,\infty)} y^{-4} \mu_N(h,dy) \leq C N^{d'}.
$$

\medskip

\noindent (iv) Furthermore, let $\mathcal{M}:=\left(m_{1,N},\ldots,m_{N,N}\right)_{N\in\Nat}$ denote the collection of all random masses. We assume that, given $H$, the collection  $\mathcal{M}$ is a family of independent random variables. Moreover,  all random variables $m_{k,N}$, $N\in\Nat$, $k=1,\ldots,N$, are independent   from   $A$ and all the  Wiener processes $(W^1_t)_{t\geq0},\ldots, (W^{n}_t)_{t\geq0},\ldots$.  
\end{assumption}

\begin{assumption}\label{assmp:m extra new}
We will use also the following  assumptions on the (conditional) distribution $\mu_N(h,\cdot)$, given $H=h$. They are either 
\begin{enumerate}
\item[(i)] There exist $\eps>0$ and $C=C(\eps,t_0)>0$ such that for all  $h\in\mathcal{H}$ the variance function $v_h(t)$ satisfies
\begin{align}\label{eq:assmp:v new}
v_h(t)\leq C t^{1+\eps},\qquad t\in[0,t_0].
\end{align}
Moreover,  there exist $\epsilon\in(0,1)$, $C=C(\epsilon)>0$  and $N_0\in\Nat$ such that for every $N>N_0$ and for each  $h\in\mathcal{H}$
\begin{align}\label{eq:assmp:muN new}
\frac{1}{m^*_N}\left(\int_{(0,\infty)}\left[ \frac{\gamma}{y} {\1}_{\{0<y\leq 1\}}+\frac{\gamma^\epsilon}{y^{2-\epsilon}}  {\1}_{\{y> 1\}} \right]\mu_N(h,dy) \right)< C.
\end{align}
\end{enumerate}
or
\begin{enumerate}
\item[(ii)] For each  $h\in\mathcal{H}$, the masses $m_{k,N}(h)$, $k=1,\ldots,N$, $N\in\Nat$, are deterministic, i.e. the (conditional) distribution $\mu_N(h,\cdot)$ is the Dirac measure concentrated at some point on $(0,\infty)$; this point may depend on {$h$} and $N$. Moreover, there exists $C=C(t_0)>0$ such that for all  $h\in\mathcal{H}$
\begin{align}
v_h(t)\leq Ct,\qquad t\in[0,t_0].
\end{align}
\end{enumerate}
\end{assumption}

\begin{assumption}\label{assmp:u0 new}
We assume that, given $H$, the collection $\left(u_0^{k,N}=u_0^{k,N}(H),\, N\in\Nat,\, k=1,\ldots,N\right)$ is a family of  independent  random variables which are independent also of   $A$ and all Wiener processes   $(W^1_t)_{t\geq0}$,..., $(W^{n}_t)_{t\geq0},\ldots$. And, given $\M$,  $u_0^{k,N}$ has Gaussian distribution $\mathcal{N}(0,\sigma/(\gamma m_{k,N}))$, $k=1,\ldots,N$.
\end{assumption}

Let now $ \myX:= {C}[0,t_0]$ be the space of all continuous functions defined on the segment $[0,t_0]$. Let $\BB(\myX)$ be the Borel sigma-field  on $\myX$.

\begin{theorem}\label{Thm2}
Fix any $t_0>0$. Under Assumptions~\ref{asmp:hyp1},~\ref{asmp:coef new},~\ref{asmp:m new},~\ref{assmp:u0 new},~\ref{assmp:coef2}, and Assumption~\ref{assmp:m extra new}~(i) or  Assumption~\ref{assmp:m extra new}~(ii) 
consider    a  process $(Z_t)_{t\in[0,t_0]}$,
\begin{align}\label{eq:new Z}
Z_t:=Z_t(A,H):=\sqrt{A}G^{(H)}_t, \qquad t\in[0,t_0],
\end{align}
where $G:\mathcal{H}\times [0,t_0]\times \Omega \rightarrow \mathbb{R}$ is assumed to be a measurable centered Gaussian field   independent of $A$ and $H$ and such that, for every $h\in \mathcal{H}$,   $(G^{(h)}_t)_{t\in[0,t_0]}$ has continuous paths and covariance function 
\begin{align*}
\Cov(  G^{(h)}_t,G^{(h)}_s)=\left( \frac{2\sigma C^2_\beta C_\delta}{\gamma^2 }\right)\frac12\left(v_h(t)+v_h(s)-v_h(|t-s|)  \right).
\end{align*}
Then the processes $\left(X^N_t\right)_{t\in[0,t_0]}$ in   system~\eqref{final system}, $N\in\Nat$, considered as $\left( \myX,\BB(\myX) \right)$-valued random elements,  converge as $N\to\infty$ to   $(Z_t)_{t\in[0,t_0]}$  in  distribution.   
\end{theorem}

\begin{proof}
This result follows immediately from Theorem~\ref{Thm}. Indeed, 
for each given  $h\in\mathcal{H}$ and each given value $a$ of $A$, the system of particles  under consideration coincides with the system of particles in the setting of Theorem~\ref{Thm}, and the corresponding solution $\left(X^N_t=X^N_t(a,h)\right)_{t\in[0,t_0]}$   is the solution  of system~\eqref{final system} in the setting of  Theorem~\ref{Thm}. So, the solution $\left(X^N_t=X^N_t(A,H)\right)_{t\in[0,t_0]}$ of system~\eqref{final system} in the setting of Theorem~\ref{Thm2}  for each given  $h\in\mathcal{H}$  and each given value $a$ of $A$ converges in distribution on the space $C[0,t_0]$ to the process 
\begin{align}\label{eq:Z(a,h)}
Z_t(a,h):=\sqrt{a}G^{(h)}_t,\qquad  t\in[0,t_0],
\end{align}
as in~\eqref{eq:new Z} by Theorem~\ref{Thm}. 
The scaling $\sqrt{a}$ in the right hand side of~\eqref{eq:Z(a,h)} appears due to Step~2 in the proof of Theorem~\ref{Thm},   since  $\left(X^N_t(a,h)\right)_{t\in[0,t_0]}$ converges in distribution to the same limit as the process $\left(\widetilde{Z}^N_t(a,h)\right)_{t\in[0,t_0]}$ (given in formula~\eqref{eq:MyZ}), where 
$\widetilde{Z}^N_t(a,h)$ has the form
$$
\widetilde{Z}^N_t(a,h)=\sqrt{a}\widehat{Z}^N_t(h),\qquad \widehat{Z}^N_t(h):= N^{a-1}\sum_{k=1}^N\beta_{k,N}\int_0^t \myU^{k,N}_\tau d\tau. 
$$
 and $\left(\widehat{Z}^N_t(h)\right)_{t\in[0,t_0]}$ converges in distribution to a centered Gaussian process with covariance function $\left( \frac{2\sigma C^2_\beta C_\delta}{\gamma^2 }\right)\frac12\left(v_h(t)+v_h(s)-v_h(|t-s|)  \right)$ by Theorem~\ref{Thm}.

Furthermore, for any continuous and bounded   mapping $f\,:\,\left(\myX,\BB(\myX)\right)\to\cR$,  we have by Theorem~\ref{Thm}, the Fubini Theorem and  the Dominated Convergence Theorem
\begin{align*}
\E[f(X^N(A,H))]&= \E\left[\E[f(X^N(A,H))\,|\,A,H]\right]=\int_0^1\left(\int_0^\infty \E[f(X^N(a,h))]\mathcal{P}_A(da)\right)\mathcal{P}_H(dh)\\
&
\longrightarrow  \int_0^1\left(\int_0^\infty \E[f(Z(a,h))]\mathcal{P}_A(da)\right)\mathcal{P}_H(dh) =\E[f(Z(A,H))],\quad N\to\infty.
\end{align*}
Here $\mathcal{P}_H$ denotes the distribution of $H$ and $\mathcal{P}_A$ denotes the distribution of $A$. Therefore, the processes $\left(X^N_t(A,H)\right)_{t\in[0,t_0]}$ in~\eqref{final system}, $N\in\Nat$, considered as $\left( \myX,\BB(\myX) \right)$-valued random elements,   converge as $N\to\infty$ to   $(Z_t(A,H))_{t\in[0,t_0]}$  in  distribution. 
\end{proof}

We  present  two  important special cases of the result of Theorem~\ref{Thm2}, which follow   from Example~\ref{ex:case1}.  For convenience of the reader  we present the whole argumentation here.

\begin{example}
For each  $h\in\mathcal{H}$ let $\nu_h$ be the L\'evy measure of a pure jump subordinator with Laplace exponent 
$\Phi_h(\lambda)=\int_{(0,\infty)} (1-e^{-\lambda y})\nu_h(dy)$.  We assume that $\nu_h$ is continuous and $\int_{(1,\infty)} y^2 \nu_h(dy)=\infty$. We assume also that
\begin{align}\label{eq:star}
\sup_{h\in \mathcal{H}}\,\,\int_{(0,\infty)}1 \wedge y\,\, \nu_h(dy)<\infty.
\end{align}
Now, fix $d,\delta>0$ and define  sequences $(m^N_{max}(h))_{N\in\Nat}$ and $(m^N_{min})_{N\in\Nat}$ via $m^N_{min}:=N^{-d}$ and $\int_{(m^N_{min},m^N_{max}(h)]} y^2\nu_h(dy)  = N^{\delta}$. For each  $h\in\mathcal{H}$ we consider the distribution $\mu_N(h,\cdot)$ given by
\begin{align}\label{eq:muN new}
\frac{d \mu_N(h,\cdot)}{d\nu_h}(y)=m^*_N y^2{\1}_{(m^N_{min},m^N_{max}(h)]}(y),
\end{align}
where the normalizing constant $m^*_N:=\left(\int_{(m^N_{min},m^N_{max}(h)]} y^2\nu_h(dy)   \right)^{-1}=N^{-\delta}$ turns $\mu_N(h,\cdot)$ into a probability measure. Then,  for any $t_0>0$ and any $t\in[0,t_0]$
\begin{equation*}
\frac{\ee^{(h)}_N(t)}{m^*_N}=  \int_{(m^N_{min},m^N_{max}(h)]} (1-e^{-\gamma yt})\; \nu_h(dy) \uparrow  \Phi_h(\gamma t)=:\dot v_h(t).
\end{equation*}
Note that $\Phi_h(\cdot)$ is  nonnegative and nondecreasing as a  Bernstein function. The limiting variance function is then
$$
v_h(t)=\int_0^t \Phi_h(\gamma \tau) d\tau.
$$ 
Due to~\eqref{eq:star}, we have
\begin{align*}
&	\int_{(0,\infty)} y^{-4} \mu_N(h,dy)= m^*_N \int_{(m^N_{min},m^N_{max}(h)]} y^{-2} \nu_h(dy) \\
&
\leq  m^*_N\left((m^N_{min})^{-3} \int_{(m^N_{min},1]} y \nu_h(dy) + \int_{(1,\infty)} \nu_h(dy)  \right) \lesssim N^{3d-\delta}.
\end{align*}
Hence, we may take $d'=3d-\delta$. And Assumption~\ref{asmp:m new}~(iii) is satisfied.  Then, Assumption~\ref{assmp:coef2} is satisfied in the following situation:
$$
b\in (0,1/2), \quad  a\in (b+1/2,(b+2/3)\wedge 1), \quad \delta=2(a-b)-1,\quad d<(4/3-2(a-b))\wedge b.
$$
 Furthermore, due to~\eqref{eq:star},  condition~\eqref{eq:assmp:muN new} of Assumption~\ref{assmp:m extra new} has then the following view:
\begin{align}\label{eq:cond nu case 1 new}
\exists\,\epsilon\in(0,1)\,\,\text{and}\,\, C=C(\epsilon)>0\,:\qquad\text{for all  } h\in\mathcal{H} \text{ we have }\int_{(1,\infty)} y^\epsilon   \nu_h(dy)\leq C.
\end{align}


Now let us consider  two  special cases of the above setting.

\bigskip

\noindent  (i) \textbf{Randomly scaled (superstatistical) fBm  with random Hurst parameter:}  Let  $H$ be a random variable  taking values in the interval $\mathcal{H}=[h_0,h^0]$ where $1/2< h_0<h^0<1$. For each  $h\in\mathcal{H}$ let $\nu_h$ be the L\'evy measure of the $(2h-1)$-stable subordinator, i.e., 
$$
\nu_h(dy):=\1_{(0,\infty)}(y)y^{-2h}dy.
$$
We define $\mu_N(h,\cdot)$ in accordance with~\eqref{eq:muN new}. 
Since
$$
\Phi_h(\lambda)=\frac{\Gamma(2-2h)}{2h-1} \lambda^{2h-1},
$$
 we find the variance function
 $$
 v_h(t)=\frac{\Gamma(2-2h) \gamma^{2h-1}}{2h(2h-1)} t^{2h}.
 $$
Hence, given $H=h$,  the limiting process $\left(G^{(h)}_t\right)_{t\in[0,t_0]}$ is (up to a multiplicative factor)  a fractional Brownian motion with Hurst parameter $h$.   Condition~\eqref{eq:assmp:v new}  is satisfied, e.g., with 
$$\eps:=2h_0-1\qquad
\text{and} 
\qquad C:=\frac{\Gamma(2-2h^0)(1\vee\gamma)}{2h_0(2h_0-1)}(1\vee t_0)^2,
$$
where $\vee$ stands for maximum. Condition~\eqref{eq:cond nu case 1 new} is satisfied with any $\epsilon\in(0, 2h_0-1)$. Hence, Assumption~\ref{assmp:m extra new}~(i) is  satisfied. Therefore, the processes 
$\left(X^N_t\right)_{t\in[0,t_0]}$ in   system~\eqref{final system}, $N\in\Nat$, considered as $\left( \myX,\BB(\myX) \right)$-valued random elements, 
converge as $N\to\infty$ in distribution to the process $(Z_t)_{t\in[0,t_0]}$, 
\begin{align*}
Z_t:=\sqrt{2\sigma C^2_\beta C_\delta \cdot A\,\frac{\Gamma(2-2H) \gamma^{2H-3}}{2H(2H-1)} } \,\, B^H_t,
\end{align*} 
where $(B^H_t)_{t\in[0,t_0]}$, given $H=h$, is a fBm with 
 Hurst parameter $h$. 
Thus, it is natural to call $(Z_t)_{t\in[0,t_0]}$ a {\it randomly scaled (or, superstatistical) fractional Brownian motion with random Hurst parameter}. Note that the 
{\it random scaling}
$$
\sqrt{2\sigma C^2_\beta C_\delta \cdot A\,\frac{\Gamma(2-2H) \gamma^{2H-3}}{2H(2H-1)} } 
$$
reflects both the individual inhomogeneity of the diffusing test-particles (via the random variable $A$) and the local 
ensemble heterogeneity of the 
surround-particles embodying the environment (via the random variable $H$). 
We remind that model based on the fBm with random Hurst parameter only has been considered in 
literature \cite{randomHurst,woszczek_etal-c-2025} 
as well as the general case
with also a random diffusion coefficient
\cite{itto_etal-jrsi-2021,lanoiselee_etal-prl-2025}.
In the case $H\equiv\const$, we are back to the model of superstatistical fBm.

\bigskip

\noindent \textbf{(ii) A randomly scaled fBm-like model with time-dependent random Hurst parameter:} Let $H=(H_1,\ldots,H_K)$ be an $\cR^K$-valued random variable, whose coordinates $H_k$, $k=1,\ldots,K$, take values in  the interval $[h_0,h^0]$, where $1/2<h_0<h^0<1$. For each value $h=(h_1,\ldots,h_K)$ of $H$ let
$$
\nu_h(dy):=\1_{(0,\infty)}(y)\sum_{k=1}^K y^{-2h_k}\,dy.
$$
We define $\mu_N(h,\cdot)$ in accordance with~\eqref{eq:muN new}. Then 
$
\Phi_h(\lambda)= \sum_{k=1}^K \frac{\Gamma(2-2h_k)}{2h_k-1} \lambda^{2h_k-1},
$
and we find the variance function
 $$
 v_h(t)=\sum_{k=1}^K \frac{\Gamma(2-2h_k) \gamma^{2h_k-1}}{2h_k(2h_k-1)} t^{2h_k}.
 $$
Hence,  given $H=h$, the limiting process $\left(G^{(h)}_t\right)_{t\in[0,t_0]}$  is a sum of  independent fBm's  with Hurst parameters $h_1,\ldots,h_K$ (up to  multiplicative factors).  Similar to the case~(i) above, condition~\eqref{eq:assmp:v new}  is satisfied, e.g., with 
$$\eps:=2h_0-1\qquad
\text{and} 
\qquad C:=K\frac{\Gamma(2-2h^0)(1\vee\gamma)}{2h_0(2h_0-1)}(1\vee t_0)^2,
$$
condition~\eqref{eq:cond nu case 1 new} is satisfied with any $\epsilon\in(0, 2h_0-1)$. Hence, Assumption~\ref{assmp:m extra new}~(i) is  satisfied. Therefore, the processes 
$\left(X^N_t\right)_{t\in[0,t_0]}$ in   system~\eqref{final system}, $N\in\Nat$, considered as $\left( \myX,\BB(\myX) \right)$-valued random elements, 
converge as $N\to\infty$ in distribution to the process $(Z_t)_{t\in[0,t_0]}$, 
\begin{align*}
Z_t:=\sqrt{2\sigma C^2_\beta C_\delta A}\,\cdot\,\sum_{k=1}^K\sqrt{\frac{\Gamma(2-2H_k) \gamma^{2H_k-3}}{2H_k(2H_k-1)}}  \,\, B^{H_k}_t,
\end{align*} 
where, given $H=h$, $(B^{H_k}_t)_{t\in[0,t_0]}$, $k=1,\ldots,K$, are independent  fractional Brownian motions with  Hurst parameters $h_k$ respectively.  
Note that we do not require neither independence, nor identical distribution of random variables  $H_1,\ldots,H_K$. 
Since our $G$, 
$$
G(h,t,\omega)\equiv G^{(h)}_t(\omega)=\sqrt{2\sigma C^2_\beta C_\delta }\,\cdot\,\sum_{k=1}^K\sqrt{\frac{\Gamma(2-2h_k) \gamma^{2h_k-3}}{2h_k(2h_k-1)}}  \,\, B^{h_k}_t(\omega),
$$ 
is a centered Gaussian field, independent of $A$ and $H$,
conditional variance of the process $(Z_t)_{t\in[0,t_0]}$, given $H$, has the following view:
\begin{align}\label{eq:cool cond variance}
  \Var(Z_t\,|\,H) = \E[Z^2_t\,|\,H]=2\sigma C^2_\beta C_\delta\E[A]\sum_{k=1}^K\frac{\Gamma(2-2H_k) \gamma^{2H_k-3}}{2H_k(2H_k-1)} t^{H_k},
\end{align}
and the total variance of the process $(Z_t)_{t\in[0,t_0]}$ is
\begin{align*}
   \Var(Z_t)&=\E[\Var(Z_t\,|\,H)]=2\sigma C^2_\beta C_\delta\gamma^{-2}\,\E[A]\int_{[h_0,h^0]^K}v_h(t)\mathcal{P}_H(dh)\\
   &
   =2\sigma C^2_\beta C_\delta\E[A]\sum_{k=1}^K\int_{h_0}^{h^0}\frac{\Gamma(2-2h_k) \gamma^{2h_k-3}}{2h_k(2h_k-1)} t^{h_k}\mathcal{P}_{H_k}(dh_k),
\end{align*}
where $\mathcal{P}_{H_k}$ is the marginal distribution of the coordinate $H_k$ of the random vector $H$. Therefore, the variance of the process $(Z_t)_{t\in[0,t_0]}$ has a complicate dependence on time $t$. If we interpret  $H$ as a description of non-uniform, site-dependent ensemble heterogeneity of the environment of the test-particle,  we may understand conditional variance $\Var(Z_t\,|\,H=h)$ of the process $(Z_t)_{t\in[0,t_0]}$, given $H=h$, as  variance of a bunch of paths of test-particles  situated in a patch of space which is characterized by the value $h=(h_1,\ldots,h_K)$ of the random vector $H$. So, locally, in this patch, behavior of  test-particles is similar to the one of the process in Example~2.1~(ii) (up to a random scaling). In particular, anomalous exponent changes with time from $2\min\{h_1,\ldots,h_K\}$ for small $t$ to $2\max\{h_1,\ldots,h_K\}$ for large $t$. In different patches of the space, corresponding to different values $h$, anomalous exponent may again vary with time and may take other values. So, this model of anomalous diffusion incorporates time-dependent random Hurst parameter (cf.~\cite{balcerek_etal-prl-2025}).
\end{example}

\begin{remark}
Other examples of Secion~\ref{sec:statement} can be  lifted to the setting of Section~\ref{SuperstatFBM} in a similar  way.
\end{remark}

\begin{remark}[\textbf{Kolmogorov--Fokker--Planck equations}]
Let $u_0\in S(\cR)$, where $S(\cR)$ is the Schwarz space of test functions. Consider
$$
u(t,x):=\E\left[ u_0\left( x+Z_t(A,H) \right) \right],\qquad t\in[0,t_0],\quad x\in\cR,
$$
where $\left( Z_t(A,H) \right)_{t\in[0,t_0]}$ is the process as in Theorem~\ref{Thm2}. Let $D:= \frac{2\sigma C^2_\beta C_\delta}{\gamma^2 }$. Then
$$
 \Var(Z_t(A,H)\,|\, A=a,\,H=h)=aD v_h(t).
$$
 Applying to $u(t,x)$ the Fourier transform $\FF$ with respect to the space variable $x$, we obtain by the Fubini Theorem
\begin{align*}
\FF[u(t,\cdot)](p)&=\int_\cR e^{-ipx}\E\left[ u_0\left( x+Z_t(A,H) \right) \right]dx\\
&
=\int_0^1\pp_H(dh)\left[  \int_0^\infty\pp_A(da)\left[  \int_\cR\E\left[ e^{-ipx}u_0(x+Z_t(a,h) \right] dx\right]\right]\\
&
=\int_0^1\pp_H(dh)\left[  \int_0^\infty\pp_A(da)\left[  \int_\cR\E\left[ e^{-ip(y-Z_t(a,h))}u_0(y) \right] dy\right]\right]\\
&
=\E\left[e^{-ADv_H(t)\frac{p^2}{2}} \right]\FF[u_0](p).
\end{align*}
Therefore, $u(t,x)=\Psi_{A,H,t}\,u_0(x)$, where $\Psi_{A,H,t}$ is a pseudo-differential operator with symbol $\E\left[\exp\left\{  -ADv_H(t)\frac{p^2}{2} \right\}\right]$. Under  suitable assumptions
on   $v_h$,
  we get
\begin{align*}
\FF[u(t,\cdot)](p)&=\FF[u_0](p)+\int_0^t \left(\frac{\pd}{\pd s}\FF[u(s,\cdot)](p)\right)ds\\
&
=\FF[u_0](p)+\int_0^t\left( \E\left[\left(-\frac{p^2}{2} AD \dot{v}_H(s) \right) e^{-ADv_H(s)\frac{p^2}{2}} \right]\FF[u_0](p)\right)ds\\
&
=\FF[u_0](p)+\int_0^t\left( \frac{\E\left[\left(-\frac{p^2}{2} AD \dot{v}_H(s) \right) e^{-ADv_H(s)\frac{p^2}{2}}\right] }{\E\left[e^{-ADv_H(s)\frac{p^2}{2}} \right]} \FF[u(s,\cdot)](p)\right)ds.
\end{align*}
Therefore, $u(t,x)=\E\left[ u_0\left( x+Z_t(A,H) \right) \right]$ solves the evolution equation
\begin{align*}
u(t,x)=u_0(x)+\int_0^t \mathcal{K}_{A,H}(s,-\Delta/2) u(s,x)ds,
\end{align*}
where $\mathcal{K}_{A,H}(s,-\Delta/2)$ is a pseudo-differential operator with symbol
\begin{align*}
\mathcal{K}_{A,H}(s,p^2/2):&=\frac{\E\left[\left(-\frac{p^2}{2} AD \dot{v}_H(s) \right) e^{-ADv_H(s)\frac{p^2}{2}}\right] }{\E\left[e^{-ADv_H(s)\frac{p^2}{2}} \right]}\\
&
=\frac{D\frac{p^2}{2}\int_0^1 \dot{v}_h(s) \mathcal{L}'[A]\left(Dv_h(s)\frac{p^2}{2}  \right)\pp_H(dh)}{\int_0^1  \mathcal{L}[A]\left(Dv_h(s)\frac{p^2}{2}  \right)\pp_H(dh)}.
\end{align*}
Here $\mathcal{L}[A](\lambda):=\E\left[ e^{-\lambda A} \right]$ is the Laplace transform of the distribution of $A$, and $\mathcal{L}'[A]$ is its derivative. Such evolution equations and their solutions have been considered in~\cite{Yana-Merten} in the special case $H\equiv\const$. For some particular choices of the distribution of $A$,   the operator $\mathcal{K}_{A,H}(s,-\Delta/2)$ coincides with some known operators of fractional calculus (see examples in Sec.~6 of~\cite{Yana-Merten} for further details).
\end{remark}

	\section{Proof of Theorem~\ref{Thm}}\label{Proofs}

	Let us start with some preparatory results. In the sequel, we use the constant $C>0$ which may change from line to line but is always independent of $N,k,M, t$ and any other time indices. We denote by $\E[\cdot\,|\,\M]$ the expectation given $\M$.
	\begin{lemma}\label{lemma:preliminary}
		Under Assumptions of Theorem~\ref{Thm}~(i) consider the Ornstein-Uhlenbeck processes given by~\eqref{eq:OUprocess}. Then we have
		\begin{align}
			&\E\left[\myU^{k,N}_t\myU^{k,N}_s\,|\,\M\right]= \frac{\sigma}{\gamma m_{k,N}}e^{-\gamma m_{k,N}|t-s|}.\label{eq:ExpCondOU}
			\end{align}
	\end{lemma}
	
	\begin{proof}
		Due to independence of $u_0^{k,N}$ and  the Wiener process $W^k$,   we have
		\begin{align*}
			&\E\left[ \myU^{k,N}_t\myU^{k,N}_s\,|\,\M\right]\\
			&
			=\E[(u_0^{k,N})^2\,|\,\M]e^{-\gamma m_{k,N}(t+s)}+2\sigma e^{-\gamma m_{k,N}(t+s)} \E\left[\left(\int_0^{t\wedge s} e^{\gamma m_{k,N}\tau}dW^k_\tau\right)^2\big|\M\right]\\
			&
			=\frac{\sigma}{\gamma m_{k,N}}e^{-\gamma m_{k,N}(t+s)}+2\sigma e^{-\gamma m_{k,N}(t+s)}\int_0^{t\wedge s} e^{2\gamma m_{k,N}\tau}d\tau\\
			&
			=\frac{\sigma}{\gamma m_{k,N}}e^{-\gamma m_{k,N}|t-s|}.
		\end{align*}
		\end{proof}

	\subsection{ Step 1: Elimination of the cross-interaction terms in equations for velocities of surrounding particles}

	\begin{lemma}\label{lemma1}
		Fix any $t_0>0$. Let $(X^N_t)_{t\in[0,t_0]}$ be the solution of  system~\eqref{final system} and $(\MyX^N_t)_{t\in[0,t_0]}$ be the process given in formula~\eqref{eq:MyX}. Under Assumptions~\ref{asmp:coef},~\ref{asmp:m},~\ref{assmp:u0},~\ref{assmp:coef2}, we have with a suitable constant $C>0$
		\begin{align}\label{eq:tilda is OK}
			\E\left[\sup\limits_{t\in[0,t_0]} \left| X^N_t -  \widetilde{X}^N_t  \right|^2   \right]\leq C v(t_0)t_0^2N^{-2(b-d)}\exp\left( C N^{-2(b-d) t_0^2} \right)\longrightarrow 0,
		\end{align} 
         as $N\to\infty$.
	\end{lemma}
	
	\begin{proof}
		We first estimate $|V^N_t-\MyV^N_t|$ for $(V^N_t)_{t\in[0,t_0]}$ and $(\MyV^N_t)_{t\in[0,t_0]}$ as in system~\eqref{final system} and formula~\eqref{eq:MyX} respectively. For this aim, we consider the following weights
		\begin{align}
			&w_\tau:=\frac1M \sum_{k=1}^N\frac{\beta_{k,N}\alpha_{k,N}}{m_{k,N}}\int_\tau^t e^{-\frac{\myA_N}{M}(t-s)}e^{-\gamma m_{k,N}(s-\tau)}ds\nonumber\\
			&
			\quad\leq \frac{1}{\myA_N}\sum_{k=1}^N\frac{\beta_{k,N}\alpha_{k,N}}{m_{k,N}}\leq C N^{-(b-d)};\label{wTau}\\
			&
			\overline{w}_\rho:=\int_\rho^t e^{-\frac{\myA_N}{M}(\tau-\rho)}w_\tau d\tau\leq C\frac{M}{\myA_N} N^{-(b-d)}.\label{wRho}
		\end{align}
		Further, for the Ornstein-Uhlenbeck process $(\myU^{k,N}_t)_{t\in[0,t_0]}$  and the process $(U^{k,N}_t)_{t\in[0,t_0]}$ as in system~\eqref{final system}, we have
		\begin{align*}
			d\left( U^{k,N}_t-\myU^{k,N}_t \right)=-\gamma m_{k,N}\left(  U^{k,N}_t-\myU^{k,N}_t \right)dt+\frac{\alpha_{k,N}}{m_{k,N}}V^N_t dt.
		\end{align*}
		And hence
		\begin{align*}
			U^{k,N}_t-\myU^{k,N}_t=\int_0^t\frac{\alpha_{k,N}}{m_{k,N}}e^{-\gamma m_{k,N}(t-\tau)}V^{N}_\tau d\tau. 
		\end{align*}
		Therefore, using the Fubini Theorem,
		\begin{align}\label{CB3}
			&V^N_t-\MyV^N_t=\frac1M\sum_{k=1}^N\beta_{k,N}\int_0^t e^{-\frac{\myA_N}{M}(t-s)}\left(U^{k,N}_s-\myU^{k,N}_s\right) ds\nonumber\\
			&
			=\frac1M\sum_{k=1}^N\beta_{k,N}\int_0^t e^{-\frac{\myA_N}{M}(t-s)}\left( \int_0^s\frac{\alpha_{k,N}}{m_{k,N}}e^{-\gamma m_{k,N}(s-\tau)}V^{N}_\tau d\tau \right)ds\nonumber\\
			&
			=\int_0^t V^N_\tau \left( \frac1M \sum_{k=1}^N\frac{\beta_{k,N}\alpha_{k,N}}{m_{k,N}}\int_\tau^t e^{-\frac{\myA_N}{M}(t-s)}e^{-\gamma m_{k,N}(s-\tau)}ds \right)d\tau\nonumber\\
			&
			=\int_0^t \left( V^N_\tau-\MyV^N_\tau \right)w_\tau d\tau+\int_0^t  \MyV^N_\tau  w_\tau d\tau.
		\end{align}
		Moreover, again using the Fubini Theorem,
		\begin{align}\label{CB4}
			&\int_0^t  \MyV^N_\tau  w_\tau d\tau = \int_0^t\left( \frac1M\sum_{k=1}^N \beta_{k,N}\int_0^\tau  e^{-\frac{\myA_N}{M}(\tau-\rho)}{ \myU^{k,N}_\rho} d\rho \right)w_\tau d\tau\nonumber\\
			&
			=\frac1M\sum_{k=1}^N \beta_{k,N}\int_0^t \myU^{k,N}_s\left(\int_\rho^te^{-\frac{\myA_N}{M}(\tau-\rho)}w_\tau d\tau  \right)d\rho\nonumber\\
			&
			=\frac1M\sum_{k=1}^N \beta_{k,N}\int_0^t \myU^{k,N}_\rho\overline{w}_\rho d\rho.
		\end{align}
		Therefore,
		\begin{align}\label{CB5}
			&\E\left[\left( \int_0^t  \MyV^N_\tau  w_\tau d\tau \right)^2\,\big|\,\M  \right]\nonumber\\
			&
			=\frac{1}{M^2}\sum_{k=1}^N\beta_{k,N}^2\int_0^t\int_0^t  \E\left[ \myU^{k,N}_\rho\myU^{k,N}_\tau\,|\,\M\right] \overline{w}_\rho \overline{w}_\tau d\rho d\tau\nonumber\\
			&
			= \frac{C}{M^2}\sum_{k=1}^N\beta_{k,N}^2\int_0^t\int_0^t \frac{ \overline{w}_\rho \overline{w}_\tau}{m_{k,N}}e^{-\gamma m_{k,N}|\rho-\tau|} d\rho d\tau\nonumber\\
			&
			\leq C N^{-2(b-d)}\frac{1}{\myA_N^2}\sum_{k=1}^N\beta_{k,N}^2\int_0^t\int_0^t \frac{ 1}{m_{k,N}}e^{-\gamma m_{k,N}|\rho-\tau|} d\rho d\tau.
		\end{align}
		Since $ 2(a-b)-1-\delta=0$ by Assumption~\ref{assmp:coef2},
		\begin{align}\label{CB6}
			&\E\left[\left( \int_0^t  \MyV^N_\tau  w_\tau d\tau \right)^2\right]\nonumber\\
			&
			\leq C N^{-2(b-d)}\frac{2}{\myA_N^2}\sum_{k=1}^N\beta_{k,N}^2\int_0^t\int_0^\tau\E\left[ \frac{ 1}{m_{k,N}}e^{-\gamma m_{k,N}(\tau-\rho)}\right] d\rho d\tau\nonumber\\
			&\leq 
			 C N^{-2(b-d)}\frac{2}{\gamma \myA_N^2}\sum_{k=1}^N\beta_{k,N}^2\int_0^t\E\left[ \frac{ 1}{ m_{k,N}^2}\left(1-e^{-\gamma m_{k,N}\tau}\right)\right] d\tau \nonumber \\ &= C N^{-2(b-d)}\frac{2}{\gamma \myA_N^2}\sum_{k=1}^N\beta_{k,N}^2 \int_0^t \ee_N(\tau)d\tau  \nonumber \\
			&
			\leq C N^{-2(b-d)}N^{-2+2a-2b+1-\delta}v(t)=C N^{-2(b-d)} v(t).
		\end{align}
		Combining~\eqref{wTau},~\eqref{CB3} and~\eqref{CB6} yields by the Fubini theorem and by the inequality $ab\leq (a^2+b^2)/2$ for any positive $a,b$
		\begin{align*}
			&\E\left[|V^N_t-\MyV^N_t|^2  \right]\leq 2\E\left[ \left| \int_0^t(V^N_\tau-\MyV^N_\tau)w_\tau d\tau  \right|^2 \right]+2\E\left[ \left|   \int_0^t \MyV^N_\tau w_\tau d\tau\right|^2 \right]\\
			&
			\leq C N^{-2(b-d)}{   t\int_0^t \E\left[ \left| V^N_\tau-\MyV^N_\tau \right|^2 \right]d\tau} +Cv(t)N^{-2(b-d)}.
		\end{align*}
Since $v$ is nondecreasing,	we have by Gr\"onwall's inequality, for every fixed $t_0>0$ and every $t\in[0,t_0]$
		\begin{align*}
			&\E\left[|V^N_t-\MyV^N_t|^2  \right]\leq Cv(t_0)N^{-2(b-d)}\exp\left( C N^{-2(b-d)}t_0^2 \right).
		\end{align*}
		Therefore, for every $t_0>0$
		\begin{align}\label{CB7}
			\sup\limits_{t\in[0,t_0]}\E\left[|V^N_t-\MyV^N_t|^2  \right]\leq C v(t_0)N^{-2(b-d)}\exp\left( C N^{-2(b-d)}t_0^2 \right).
		\end{align}
		Finally,
		\begin{align*}
			&\E\left[\sup\limits_{t\in[0,t_0]}|X^N_t-\MyX^N_t|^2  \right]=\E\left[\sup\limits_{t\in[0,t_0]}\left|\int_0^t ( V^N_s-\MyV^N_s) ds\right|^2  \right]\\
			&
			\leq t_0\int_0^{t_0} \E\left[|V^N_s-\MyV^N_s|^2  \right]ds\leq t^2_0\sup\limits_{s\in[0,t_0]}\E\left[|V^N_s-\MyV^N_s|^2  \right].
		\end{align*}
		Thus, by~\eqref{CB7} and Assumption~\ref{assmp:coef2},
		\begin{align*}
			\E\left[\sup\limits_{t\in[0,t_0]}|X^N_t-\MyX^N_t|^2  \right]\leq C v(t_0)t_0^2N^{-2(b-d)}\exp\left(C N^{-2(b-d)}t_0^2  \right)\longrightarrow 0,
		\end{align*}
         as $N\to\infty$.
	\end{proof}

	\subsection{ Step 2: Comparison of  $\left(\MyX^N_t\right)_{t\in[0,t_0]}$  with $\left(\MyZ^N_t\right)_{t\in[0,t_0]}$.}
	
	\begin{lemma}\label{lemma2}
		For each fixed $t_0>0$ consider $(\MyX^N_t)_{t\in[0,t_0]}$ and $(\MyZ^N_t)_{t\in[0,t_0]}$ as in formulas~\eqref{eq:MyX} and~\eqref{eq:MyZ} respectively. Under assumptions of Theorem~\ref{Thm}~(i), we have
		\begin{align*}
			\sup\limits_{t\in[0,t_0]} \E\left[|\MyX^N_t-\MyZ^N_t|^2  \right]\leq C M \dot v(t_0) N^{a-1}\longrightarrow 0,\qquad N\to\infty.
		\end{align*}
	\end{lemma}
	
	\begin{proof}
		We have by formula~\eqref{eq:MyX} and by the Fubini theorem
		\begin{align*}
			\MyX^N_t &=\frac1M \sum_{k=1}^N\beta_{k,N}\int_0^t e^{-\frac{\myA_N}{M}s}\int_s^t \myU^{k,N}_{\tau-s} d\tau ds\\
			&
			=\frac1M \sum_{k=1}^N\beta_{k,N}\int_0^t e^{-\frac{\myA_N}{M}s}\int_0^{t-s} \myU^{k,N}_{\tau} d\tau ds\\
			&
			=\frac1M \sum_{k=1}^N\beta_{k,N}\int_0^t \myU^{k,N}_\tau \int_0^{t-\tau}e^{-\frac{\myA_N}{M}s}ds d\tau\\
			&
			=\frac{1}{\myA_N} \sum_{k=1}^N\beta_{k,N}\int_0^t \myU^{k,N}_\tau \left( 1-e^{-\frac{\myA_N}{M}(t-\tau)} \right) d\tau\\
			&
			=\MyZ^N_t-\frac{1}{\myA_N} \sum_{k=1}^N\beta_{k,N}\int_0^t \myU^{k,N}_\tau e^{-\frac{\myA_N}{M}(t-\tau)}  d\tau.
		\end{align*}
		Consider $\mathcal{R}^N_t:=\frac{1}{\myA_N} \sum_{k=1}^N\beta_{k,N}\int_0^t \myU^{k,N}_\tau e^{-\frac{\myA_N}{M}(t-\tau)}  d\tau$. Hence, by~\eqref{eq:ExpCondOU} and  Assumptions~\ref{asmp:m}, \ref{assmp:coef2},
		\begin{align*}
			&\E\left[\left| \MyX^N_t-\MyZ^N_t  \right|^2  \right]=\E\left[\left|  \mathcal{R}^N_t \right|^2  \right]\\
			&
			=\frac{1}{\myA^2_N}\sum_{k=1}^N\beta^2_{k,N}\int_0^t\int_0^t e^{-\frac{\myA_N}{M}(t-\tau)}e^{-\frac{\myA_N}{M}(t-\rho)}\E\left[ \myU^{k,N}_\tau\myU^{k,N}_\rho \right] d\rho d\tau \\
			&
			=\frac{2}{\myA^2_N}\sum_{k=1}^N\beta^2_{k,N}\int_0^t\int_0^\tau e^{-\frac{\myA_N}{M}(t-\tau)}e^{-\frac{\myA_N}{M}(t-\rho)}\E\left[ \frac{\sigma}{\gamma m_{k,N}}e^{-\gamma m_{k,N}(\rho-\tau)} \right] d\rho d\tau\\
						&
			\leq { \frac{2\sigma}{\gamma^2\myA^2_N}}\sum_{k=1}^N\beta^2_{k,N}\int_0^t e^{-\frac{\myA_N}{M}(t-\tau)} \ee_N(\tau )d\tau   \\
			&
			\leq C N^{-2+2a+1-2b-\delta} \dot v(t_0)\int_0^t e^{-\frac{\myA_N}{M}(t-\tau)} d\tau\\
			&
			\leq C \dot v(t_0) \frac{M}{\myA_N}\leq C M \dot v(t_0) N^{a-1}\longrightarrow0,\quad N\to\infty.
		\end{align*}
	\end{proof}

	\bigskip
	

	\subsection{ Step 3: Convergence to  the Gaussian Process $\left( Z_t \right)_{t\in[0,t_0]}$ }
	
	Let us introduce the following notations:
	\begin{align*}
		&\ffi^{t,s}(m):=\int_0^t\int_0^s\frac1m e^{-\gamma m|\tau-\rho|}d\tau d\rho;\\
		&
		\xi^{N,t,s}_k:=\frac{\sigma}{\gamma \myA_N^2}\beta^2_{k,N}\ffi^{t,s}(m_{k,N});\\
		&
		\eta^{N,t,s}_k:=\xi^{N,t,s}_k-\E\left[ \xi^{N,t,s}_k \right];\\
		&
		\eta^{N,t,s}:=\sum_{k=1}^N\eta^{N,t,s}_k;\\
		&
		\xi^{N,t,s}:=\sum_{k=1}^N \xi^{N,t,s}_k=\Cov(\MyZ^N_t,\MyZ^N_s\,|\,\M),
	\end{align*}
	  where the last identity is due to Lemma~\ref{lemma:preliminary} and \eqref{eq:MyZ}.
	\begin{lemma}\label{expZ}
		Under assumptions of Theorem~\ref{Thm}~(i), we have 
		\begin{align*}
			\E\left[ \xi^{N,t,s} \right]\longrightarrow \left(\frac{2\sigma C_\beta^2C_{\delta}}{\gamma^2 C_\alpha^2}\right) \frac{1}{2}\left(v(t)+v(s)-v(|t-s|) \right)  ,\qquad N\to\infty.
		\end{align*}
	\end{lemma}
	
	\begin{proof}
Let 
$$
\bar Z^{k,N}_t:=\int_0^t \widetilde U^{k,N}_\tau d\tau.
$$
Since $\widetilde U^{k,N}$ is (conditionally on $\mathcal{M}$) a stationary Ornstein-Uhlenbeck process, $\bar Z^{k,N}$ has stationary increments conditionally  on $\mathcal{M}$. Hence, 
using the   equality $ab=\frac12\left(a^2+b^2-(a-b)^2\right)$ for $a,b>0$, we obtain

\begin{align}\label{eq:cond_cov}
&\Cov(\MyZ^N_t,\MyZ^N_s\,|\,\M)=\frac{1}{\myA_N^2}\sum_{k=1}^N\beta_{k,N}^2 \E[\bar Z^{k,N}_t \bar Z^{k,N}_s|\mathcal{M}] \nonumber  \\
 &
 = \frac{1}{2\myA_N^2}\sum_{k=1}^N\left(\beta_{k,N}^2 \left(v_{k,N}(t)+v_{k,N}(s)-v_{k,N}(|t-s|) \right)\right),
\end{align}
	where $v_{k,N}(t):=\E[(\bar Z^{k,N}_t)^2|\mathcal{M}]$. Now, by Lemma~\ref{lemma:preliminary},
\begin{align*}
v_{k,N}(t)&= 2 \int_0^t \int_0^\tau  \E[\widetilde U^{k,N}_\tau\widetilde U^{k,N}_\rho|\mathcal{M}]d\rho d\tau \nonumber \\
&
=   \int_0^t \int_0^\tau  \frac{2\sigma}{\gamma m_{k,N}}e^{-\gamma m_{k,N}(\tau-\rho)} d\rho d\tau \nonumber \\
&
= \frac{2\sigma}{\gamma^2} \int_0^t \frac{1-e^{-\gamma m_{k,N}\tau}}{m_{k,N}^2}d\tau.
\end{align*}
	Taking expectation and applying Fubini's theorem and the monotone convergence theorem, we obtain
	\begin{equation}\label{eq:exp_v}
		N^\delta \E[v_{k,N}(t)]=\frac{2\sigma}{\gamma^2}  N^\delta m^*_N \int_0^t  \frac{\ee_N(\tau )}{m^*_N} d\tau \rightarrow \frac{2\sigma C_{\delta}}{\gamma^2}v(t).
	\end{equation}
In view of  \eqref{eq:cond_cov},
$$
\Cov(\MyZ^N_t,\MyZ^N_s)\rightarrow \frac{2\sigma C_\beta^2 C_\delta}{\gamma^2 C_\alpha^2}\; \frac{1}{2}\left(v(t)+v(s)-v(|t-s|) \right). 
$$
	\end{proof}
	
	\begin{lemma}\label{4th moment}
		Under assumptions of Theorem~\ref{Thm}~(i) we have
		\begin{align*}
			{ \sum_{N=1}^\infty}\E\left[|\eta^{N,t,s}|^4 \right]<\infty.
		\end{align*}
	\end{lemma}
	\begin{proof}
		Since $\E\left[ \eta^{N,t,s}_k \right]=0$, we obtain by the i.i.d. property of masses $m_{1,N},\ldots,m_{N,N}$
		\begin{align*}
			\E\left[|\eta^{N,t,s}|^4 \right]=N\E\left[|\eta^{N,t,s}_1|^4 \right]+3N(N-1)\E\left[|\eta^{N,t,s}_1|^2 \right]^2.
		\end{align*}
		Hence
		\begin{align}\label{CB*}
			\E\left[|\eta^{N,t,s}|^4 \right]\leq 3N^2 \E\left[ \left(\xi^{N,t,s}_1-\E\left[ \xi^{N,t,s}_1 \right]\right)^4 \right]\leq C N^2\E\left[\left| \xi^{N,t,s}_1 \right|^4 \right].
		\end{align}
				Now,
		\begin{align*}
			\left| \xi^{N,t,s}_1 \right|^4 \leq C N^{-8(1+b-a)}\left| \ffi^{t,s}(m_{1,N} )\right|^4.
		\end{align*}
	Noting that
	\begin{align*}
		\ffi^{t,s}(m)=\int_0^t\int_0^s\frac1m e^{-\gamma m|\tau-\rho|}d\tau d\rho\leq \frac{ts}{m},
	\end{align*}
	we observe,
			\begin{align*}
			\E\left[ \left| \xi^{N,t,s}_1 \right|^4 \right]\leq C t_0^8 N^{-8(1+b-a)} \int_{(0,\infty)} y^{-4} \mu_N(dy)\leq C t_0^8 N^{-8(1+b-a)+d'}
		\end{align*}
		
		Therefore,
		\begin{align*}
			{ \sum_{N=1}^\infty} \E\left[|\eta^{N,t,s}|^4 \right]\leq C t_0^8 { \sum_{N=1}^\infty}  N^{-1+(8a-8b-5+d')}<\infty,
		\end{align*}
		since $8a-8b-5+d'<0$ due to Assumption~\ref{assmp:coef2}.
	\end{proof}
	
	\begin{lemma}\label{lem:Step3}
		Under assumptions of Theorem~\ref{Thm}~(i), we have $\PP$-almost surely 
		\begin{align*}
			\Cov(\MyZ^N_t,\MyZ^N_s\,|\,\M)\longrightarrow \left(\frac{2\sigma C_\beta^2C_{\delta}}{\gamma^2 C_\alpha^2}\right) \frac{1}{2}\left(v(t)+v(s)-v(|t-s|) \right) ,\qquad N\to\infty.
		\end{align*}
	\end{lemma}
	
	\begin{proof}
		For every $\eps>0$, we have by Lemma~\ref{4th moment} and Markov's inequality
		\begin{align*}
			\sum_{N=1}^\infty \PP\left( \left\{  \left| \eta^{N,t,s} \right|>\eps \right\}  \right)\leq \frac{1}{\eps^4}  \sum_{N=1}^\infty\E\left[|\eta^{N,t,s}|^4 \right]<\infty.
		\end{align*}
		Hence, $\eta^{N,t,s}\to0$, $N\to\infty$, $\PP$-almost surely by the Borel--Cantelli Lemma. Therefore, the statement follows from  Lemma~\ref{expZ}. 
	\end{proof}
	
Let us now recall the notion of mixing convergence of random elements. 
	\begin{definition}
		Let $(\Omega,\FF,\PP)$ be a probability space.  Let $\myX$  be a separable metrizable topological space endowed with its Borel $\sigma$-field $\BB(\myX)$. Let $\GG\subset\FF$ be a sub-$\sigma$-field. A sequence $(\xi_N)_{N\in\Nat}$ of $(\myX,\BB(\myX))$-valued random elements is said to \emph{converge $\GG$-mixing} to a (random element with) probability distribution $\nu$ on $\myX$, if the conditional distributions $\mathcal{P}_{\xi_N|\GG}$ converge weakly to $\nu$ as $N\to\infty$, i.e., if for every $f\in L^1(\Omega,\FF,\PP)$ and every bounded continuous function $h$ on $\myX$
		\begin{align*}
			\lim_{N\to\infty}\E\left[ f\E\left[h(\xi_N)|\GG  \right] \right] =\E[f]\int_\myX h d\nu.
		\end{align*}
	\end{definition}
\begin{remark}\label{rem:mixing}
	Note that the  $\GG$-mixing convergence is a special case of the so called  $\GG$-stable convergence;   and $\GG$-stable convergence implies also convergence in distribution (see~\cite{zbMATH06444973} for further details).
	\end{remark}
	
	\begin{lemma}\label{lem:FinDimZ}
		Let $\sigma(\M)$ be the $\sigma$-field generated by the family $\M$ given in Assumption~\ref{asmp:m}. In the setting of Theorem~\ref{Thm}~(i),  the sequence of stochastic processes $\left((\MyZ^N_t)_{t\in[0,t_0]}\right)_{N\in\Nat}$  converges $\sigma(\M)$-mixing in finite dimensional distributions to the process  $(Z_t)_{t\geq0}$, i.e., for any $n\in\Nat$ and any $t_1,\ldots,t_n\in[0,t_0]$, the $(\cR^n,\BB(\cR^n))$-valued random elements $(\MyZ^N_{t_1},\ldots,\MyZ^N_{t_n})$ converge $\sigma(\M)$-mixing to $(Z_{t_1},\ldots,Z_{t_n})$.
	\end{lemma}
	\begin{proof}
		By Lemma~\ref{lem:Step3}, conditionally on  $\M$, we have a sequence of Gaussian stochastic processes $\left((\MyZ^N_t)_{t\in[0,t_0]}\right)_{N\in\Nat}$  with continuous paths whose mean functions and covariance functions converge to those of the process  $(Z_t)_{t\in[0,t_0]}$. Therefore, conditionally on $\M$, all finite dimensional marginal distributions of processes $(\MyZ^N_t)_{t\in[0,t_0]}$  converge to those of the process  $(Z_t)_{t\geq0}$.
	\end{proof}
	
	\begin{lemma}\label{lem:finDimDistrConv}
In the setting of Theorem~\ref{Thm}~(i),  the sequence of 
processes $\left((X^N_t)_{t\in[0,t_0]}\right)_{N\in\Nat}$  converges $\sigma(\M)$-mixing   in finite dimensional distributions to the process $(Z_t)_{t\in[0,t_0]}$.
	\end{lemma}
	
	\begin{proof}
		Since, by Lemma~\ref{lem:FinDimZ},  the sequence of processes $\left((\MyZ^N_t)_{t\in[0,t_0]}\right)_{N\in\Nat}$  converges $\sigma(\M)$-mixing in finite dimensional distributions  to the process  $(Z_t)_{t\in[0,t_0]}$, it is enough to show that for any $\eps>0$,  any $n\in\Nat$ and any $t_1,\ldots,t_n\in[0,t_0]$
		\begin{align*}
			\PP\left(\left\|(X^N_{t_1},\ldots,X^N_{t_n})-  (\MyZ^N_{t_1},\ldots,\MyZ^N_{t_n})\right\|_{\cR^n} >\eps \right)\longrightarrow 0,\qquad N\to\infty,
		\end{align*}
		and to apply Theorem~3.7~(a) of~\cite{zbMATH06444973}. By Markov's inequality, Lemma~\ref{lemma1} and Lemma~\ref{lemma2}, we have
		\begin{align*}
			&\PP\left(\left\|(X^N_{t_1},\ldots,X^N_{t_n})-  (\MyZ^N_{t_1},\ldots,\MyZ^N_{t_n})\right\|_{\cR^n} >\eps \right)\\
			&
			\leq\frac{1}{\eps^2}\E\left[ \left\|(X^N_{t_1},\ldots,X^N_{t_n})-  (\MyZ^N_{t_1},\ldots,\MyZ^N_{t_n})\right\|_{\cR^n}^2 \right]\\
			&
			\leq \frac{2}{\eps^2}\sum_{j=1}^n\left(\E\left[|X^N_{t_j}-\MyX^N_{t_j} |^2 \right]+ \E\left[|\MyX^N_{t_j}-\MyZ^N_{t_j} |^2 \right] \right)\\
			&
			\leq \frac{2n}{\eps^2}\left( \sup_{t\in[0,t_0]} \E\left[|X^N_{t}-\MyX^N_{t} |^2 \right]+ \sup_{t\in[0,t_0]} \E\left[|\MyX^N_{t}-\MyZ^N_{t} |^2 \right]\right)\longrightarrow0,
		\end{align*}
        { as $N\to \infty$.}
	\end{proof}

Let us now refine { the previous result to} the convergence of processes $\left((X^N_t)_{t\in[0,t_0]}\right)_{N\in\Nat}$ to the process $(Z_t)_{t\in[0,t_0]}$. To this aim, we need the following preparatory result:

\begin{lemma}\label{lem:tightness}
Fix $N_0$ sufficiently large such that  $\myA_N\geq M$ for every $N\geq N_0$. Then, there is a constant $C>0$ such that for every $0\leq \epsilon\leq 1$, $N\geq N_0$ and $0\leq s\leq t\leq t_0$
\begin{align*}
&\E\left[|\widetilde X^N_t-	\widetilde X^N_s|^2\right] \\
&
\leq   C \Biggl(v(|t-s|) + \Bigl(\frac{1}{m^*_N}\int\limits_{(0,\infty)}\left[ \frac{\gamma}{y} {\1}_{\{0<y\leq 1\}}+\frac{\gamma^\epsilon}{y^{2-\epsilon}}  {\1}_{\{y> 1\}}\right] \mu_N(dy) \Bigr) \;|t-s|^{1+\epsilon} \Biggr)
\end{align*}
\end{lemma}

\begin{proof}
	Let $0\leq s\leq t \leq T$. We write (cp. the proof of Lemma \ref{lemma2})
	\begin{eqnarray*}
		&&	\widetilde X^N_t-	\widetilde X^N_s \\ &=&  \frac{1}{\myA_N} \sum_{k=1}^N\beta_{k,N}\int_s^t \myU^{k,N}_\tau \left( 1-e^{-\frac{\myA_N}{M}(t-\tau)} \right) d\tau \\ &&+ \frac{1}{\myA_N} \sum_{k=1}^N\beta_{k,N}\int_0^s \myU^{k,N}_\tau \left( e^{-\frac{\myA_N}{M}(s-\tau)}-e^{-\frac{\myA_N}{M}(t-\tau)} \right) d\tau  \\
		&=& (I)+(II)
	\end{eqnarray*}
	By standard calculations, making use of \eqref{eq:ExpCondOU}, 
	\begin{eqnarray}\label{eq:tight1}
		&& \E[|(I)|^2]\nonumber \\ &=& \left(\frac{2\sigma}{\gamma \myA_N^2} \sum_{k=1}^N\beta_{k,N}^2\right)\int_s^t \int_s^\tau \int_{(0,\infty)} \frac{e^{-\gamma y(\tau-\rho)}}{y}(1-e^{-\frac{\myA_N}{M}(t-\rho)})\nonumber \\ &&\quad \times(1-e^{-\frac{\myA_N}{M}(t-\tau)})  \mu_N(dy)d\rho d\tau	\nonumber \\ &\leq & \left(\frac{2\sigma}{\gamma \myA_N^2} \sum_{k=1}^N\beta_{k,N}^2\right)\int_s^t  \int_{(0,\infty)} \int_s^\tau \frac{e^{-\gamma y(\tau-\rho)}}{y}  d\rho\, \mu_N(dy) \, d\tau	\nonumber \\ &=& 
		\left(\frac{2\sigma}{\gamma^2 \myA_N^2} \sum_{k=1}^N\beta_{k,N}^2\right)\int_s^t  \int_{(0,\infty)}  (1-e^{-\gamma y(\tau-s)})y^{-2} \, \mu_N(dy) \, d\tau	\nonumber \\ &=& 
		\left(\frac{2\sigma}{\gamma^2 \myA_N^2} \sum_{k=1}^N\beta_{k,N}^2\right)\int_0^{t-s}  \int_{(0,\infty)}  (1-e^{-\gamma y\tau})y^{-2}  \, \mu_N(dy) \, d\tau	\nonumber \\
		&=& 	\left(\frac{2\sigma m^*_N}{\gamma^2 \myA_N^2}\sum_{k=1}^N\beta_{k,N}^2\right)\int_0^{t-s} 	\frac{\ee_N(\tau)}{m^*_N} \, d\tau \leq  	\left(\frac{2\sigma m^*_N}{\gamma^2 \myA_N^2}\sum_{k=1}^N\beta_{k,N}^2\right) v(|t-s|).\nonumber \\ \
	\end{eqnarray}
	For the estimate of $(II)$, we fix some $0\leq \epsilon\leq 1$ and assume that $\myA_N\geq M$. Moreover, we abbreviate 
	$$
	g(t,s;\tau)=e^{-\frac{\myA_N}{M}(s-\tau)}-e^{-\frac{\myA_N}{M}(t-\tau)}
	$$
	Then, 
	\begin{eqnarray*}
		&& \E[|(II)|^2]\nonumber \\ &=& \left(\frac{2\sigma}{\gamma \myA_N^2} \sum_{k=1}^N\beta_{k,N}^2\right)\int_0^s  \int_{(0,\infty)} \int_0^\tau \frac{e^{-\gamma y(\tau-\rho)}}{y}g(t,s;\rho) d\rho\, \mu_N(dy)\, g(t,s;\tau) \,d\tau	\end{eqnarray*}
Now,
\begin{eqnarray*}
&&	\int_0^\tau \frac{e^{-\gamma y(\tau-\rho)}}{y}g(t,s;\rho) d\rho
\\
&=&\frac{1}{y(\gamma y+\frac{\myA_N}{M})}\left(e^{-\frac{\myA_N}{M}(s-\tau)}-e^{-\frac{\myA_N}{M}(t-\tau)}+( e^{-\frac{\myA_N}{M}t}-e^{-\frac{\myA_N}{M}s}) e^{ -\gamma y \tau}  \right) 
\\ 
&\leq &
\frac{1}{y(\gamma y+\frac{\myA_N}{M})}\left(e^{-\frac{\myA_N}{M}(s-\tau)}-e^{-\frac{\myA_N}{M}(t-\tau)}  \right)^{1-\epsilon}
		\left(e^{-\frac{\myA_N}{M}(s-\tau)}-e^{-\frac{\myA_N}{M}(t-\tau)}  \right)^{\epsilon} 
\\ 
&\leq &
{ \frac{1}{y} \frac{1}{(\gamma y+\frac{\myA_N}{M})^\epsilon} \frac{1}{(\gamma y+\frac{\myA_N}{M})^{1-\epsilon}} }
		\left(e^{-\frac{\myA_N}{M}(s-\tau)}-e^{-\frac{\myA_N}{M}(t-\tau)}  \right)^{\epsilon} 
\\
&\leq& \left(\frac{M}{y\myA_N} {\1}_{\{0<y\leq 1\}}+\frac{M^\epsilon}{y^{2-\epsilon}\gamma^{1-\epsilon}\myA_N^\epsilon}  {\1}_{\{y> 1\}}  \right) \left(e^{-\frac{\myA_N}{M}(s-\tau)}-e^{-\frac{\myA_N}{M}(t-\tau)}  \right)^{\epsilon} 
\\
&\leq & 
\left(\frac{1}{y} {\1}_{\{0<y\leq 1\}}+\frac{1}{y^{2-\epsilon}\gamma^{1-\epsilon}}  {\1}_{\{y> 1\}}  \right) \left(\frac{M}{\myA_N}\left(e^{-\frac{\myA_N}{M}(s-\tau)}-e^{-\frac{\myA_N}{M}(t-\tau)} \right) \right)^{\epsilon} 
\\
&\leq & 
\left(\frac{1}{y} {\1}_{\{0<y\leq 1\}}+\frac{1}{y^{2-\epsilon}\gamma^{1-\epsilon}}  {\1}_{\{y> 1\}}  \right) |t-s|^\epsilon.
	\end{eqnarray*}
	Hence,
	\begin{eqnarray*}
		&& \E[|(II)|^2]\nonumber \\  &{ \leq}& \left(\frac{2\sigma}{\gamma \myA_N^2} \sum_{k=1}^N\beta_{k,N}^2\right) \left(\,\int\limits_{(0,\infty)} \left[\frac{1}{y} {\1}_{\{0<y\leq 1\}}+\frac{1}{y^{2-\epsilon}\gamma^{1-\epsilon}}  {\1}_{\{y> 1\}} \right]\mu_N(dy) \right) \nonumber \\ && \times \left(\int_0^s g(t,s,\tau) d\tau \right) |t-s|^\epsilon  
	\end{eqnarray*}
	Noting that 
	\begin{eqnarray*}
		\int_0^s g(t,s,\tau) d\tau &\leq & \frac{M}{\myA_N}	\left(1- e^{-\frac{\myA_N}{M}(t-s)}\right)\leq |t-s|,
	\end{eqnarray*}
	we finally obtain

\begin{align*}
& \E[|(II)|^2]\nonumber 
 \leq\left(\frac{2\sigma m^*_N}{\gamma \myA_N^2} \sum_{k=1}^N\beta_{k,N}^2\right)  \\ &\quad \times\left(\frac{1}{m^*_N}\,\int\limits_{(0,\infty)}\left[ \frac{1}{y} {\1}_{\{0<y\leq 1\}}+\frac{1}{y^{2-\epsilon}\gamma^{1-\epsilon}}  {\1}_{\{y> 1\}} \right]\mu_N(dy) \right) |t-s|^{1+\epsilon}.  \nonumber   
\end{align*}
Noting that the factor $m^*_N \myA_N^{-2} \sum_{k=1}^N\beta_{k,N}^2$ is bounded in $N$ by Assumption~\ref{assmp:coef2}, the proof is finished.
\end{proof}

\begin{lemma}
Let $\myX$ be the space of continuous functions $ C[0,t_0]$ endowed with its Borel $\sigma$-field $\BB(\myX)$. In the setting of Theorem~\ref{Thm}~(ii),  the sequence of processes $\left((X^N_t)_{t\in[0,t_0]}\right)_{N\in\Nat}$, considered as $(\myX,\BB(\myX))$-valued random elements, converges $\sigma(\M)$-mixing   in  distribution to the process $(Z_t)_{t\in[0,t_0]}$.
	\end{lemma}
	
	\begin{proof}
Note that each process $ (X^N_t)_{t\in[0,t_0]}$, ${N\in\Nat}$, as well as each process $ (\MyX^N_t)_{t\in[0,t_0]}$, ${N\in\Nat}$, has continuous paths as (a part of) a solution of  a multidimensional linear in a narrow sense stochastic differential equation driven by a (multidimensional) Wiener process. Further, let us first show that the sequence of processes $\left((\MyX^N_t)_{t\in[0,t_0]}\right)_{N\in\Nat}$  is tight in $\myX$. By Corollary~14.9 of~\cite{Kallenberg}, it is enough to check that $(\MyX^N_0)_{N\in\Nat}$ is tight in $\cR$ and $\E\left[|\MyX^N_t-\MyX^N_s|^q \right]\leq C|t-s|^{p}$ for some $p>1$, $q,C>0$ not depending on $N$. 		Recall that $\MyX^N_0=0$ for all $N\in\Nat$,  hence $(\MyX^N_0)_{N\in\Nat}$  is  tight. 

Consider  the setting of Assumption~\ref{assmp:m extra}~(i). We have by Lemma~\ref{lem:tightness} (with $\eps$, $\epsilon$ and $N_0$ as in Assumption~\ref{assmp:m extra}~(i)) 
\begin{align*}
&\E\left[|\widetilde X^N_t-	\widetilde X^N_s|^2\right] \leq   C |t-s|^{1+\epsilon\wedge\eps},\qquad t,s\in[0,t_0],\quad N> N_0.
\end{align*}
Hence, the family 	 $\left((\MyX^N_t)_{t\in[0,t_0]}\right)_{N\in\Nat}$ is tight in $\myX$.	

In the setting of Assumption~\ref{assmp:m extra}~(ii), we have by Lemma~\ref{lem:tightness} with $\epsilon:=0$
\begin{align*}
&\E\left[|\widetilde X^N_t-	\widetilde X^N_s|^2\right] \leq   C |t-s|,\qquad t,s\in[0,t_0],\quad N> N_0.
\end{align*}		
Since, in the setting of Assumption~\ref{assmp:m extra}~(ii),  $\widetilde X^N_t-	\widetilde X^N_s$ is Gaussian (and not only conditionally Gaussian as in the case of random masses $m_{k,N}$), we get,
$$
\E\left[|\widetilde X^N_t-	\widetilde X^N_s|^4\right]\leq 3 C^2 |t-s|^2.
$$ 
Therefore, the family 	 $\left((\MyX^N_t)_{t\in[0,t_0]}\right)_{N\in\Nat}$ is again tight in $\myX$.

Further, it follows from the proof of Lemma~\ref{lem:finDimDistrConv}  that the sequence of stochastic processes $\left((\MyX^N_t)_{t\in[0,t_0]}\right)_{N\in\Nat}$  converges $\sigma(\M)$-mixing   in finite dimensional distributions to the process $(Z_t)_{t\in[0,t_0]}$. Therefore, the processes $\left((\MyX^N_t)_{t\in[0,t_0]}\right)_{N\in\Nat}$  converge    to the process $(Z_t)_{t\in[0,t_0]}$ $\sigma(\M)$-mixing  in  distribution by  Proposition~3.9 of~\cite{zbMATH06444973}. 

Moreover, we have for any $\eps>0$ by Markov's inequality
\begin{align*}
&\PP\left(\left\| X^N-\MyX^N   \right\|_{C([0,t_0])} >\eps \right)\leq \frac{1}{\eps^2}\E\left[\left\| X^N-\MyX^N   \right\|^2_{C([0,t_0])}  \right]\\
&
=\frac{1}{\eps^2}\E\left[\sup_{t\in[0,t_0]}\left| X^N_t-\MyX^N_t   \right|^2  \right]\longrightarrow 0,\qquad N\to\infty,
\end{align*}
by Lemma~\ref{lemma1}. Therefore, 	the sequence of processes $\left((X^N_t)_{t\in[0,t_0]}\right)_{N\in\Nat}$ converges $\sigma(\M)$-mixing   in  distribution to the process $(Z_t)_{t\in[0,t_0]}$ by  Theorem~3.7~(a) of~\cite{zbMATH06444973}. In particular, the sequence of processes $\left((X^N_t)_{t\in[0,t_0]}\right)_{N\in\Nat}$ converges  to the process $(Z_t)_{t\in[0,t_0]}$   in  distribution (cp. Remark~\ref{rem:mixing}).
\end{proof}

\section{Acknowledgments}
Gianni Pagnini acknowledges the support by the Basque Government through the BERC 2022--2025 
program and by the Ministry of Science and Innovation: BCAM Severo Ochoa accreditation\\ CEX2021-001142-S / MICIN / AEI / 10.13039/501100011033.	

Yana Kinderknecht acknowledges the  Basque Center of Applied Mathematics (BCAM) and Kassel University for the financial support of  her regular research visits to BCAM.
	

\end{document}